\documentclass[10pt]{article}
\usepackage{graphicx}
\usepackage{amssymb,amsmath}
\usepackage[square,numbers,sort]{natbib}
\usepackage{hyperref,xcolor}
\usepackage[margin=1in]{geometry}

\newtheorem{theorem}{Theorem}
\newtheorem{lemma}[theorem]{Lemma}

\newtheorem{corollary}[theorem]{Corollary}
\newtheorem{example}{Example}
\newtheorem{definition}{Definition}

\newenvironment{proof}[1][Proof]{\begin{trivlist}
\item[\hskip \labelsep {\bfseries #1}]}{\end{trivlist}}
\newenvironment{remark}[1][Remark]{\begin{trivlist}
\item[\hskip \labelsep {\bfseries #1}]}{\end{trivlist}}

\newcommand{\qed}{$\square$}

\newcommand\lmax{l_\text{max}}
\newcommand\cost{\text{COST}}

\begin{document}
\author{
Hadi Pouransari\footnote{email: hadip@stanford.edu} {\scriptsize AND} Eric Darve
}
\title{OPTIMIZING THE ADAPTIVE FAST MULTIPOLE METHOD FOR FRACTAL SETS}
\date{}
\maketitle


\section* {Abstract}
We have performed a detailed analysis of the fast multipole method (FMM) in the adaptive case, in which the depth of the FMM tree is non-uniform. Previous works in this area have focused mostly on special types of adaptive distributions, for example when points accumulate on a 2D manifold or accumulate around a few points in space. Instead, we considered a more general situation in which fractal sets, e.g., Cantor sets and generalizations, are used to create adaptive sets of points. Such sets are characterized by their dimension, a number between 0 and 3. We introduced a mathematical framework to define a converging sequence of octrees, and based on that, demonstrated how to increase $N \to \infty$.

A new complexity analysis for the adaptive FMM is introduced. It is shown that the ${\cal{O}}(N)$ complexity is achievable for any distribution of particles, when a modified adaptive FMM is exploited. We analyzed how the FMM performs for fractal point distributions, and how optimal parameters can be picked, e.g., the criterion used to stop the subdivision of an FMM cell. A new subdividing double-threshold method is introduced, and better performance demonstrated. Parameters in the FMM are modeled as a function of particle distribution dimension, and the optimal values are obtained. A three dimensional kernel independent black box adaptive FMM is implemented and used for all calculations. 

\section* {Key words}
Adaptive fast multipole method, fractal set, octree, linear complexity


\section {Introduction}

The N-body problem, notwithstanding its well established analytical difficulty \cite{strogatz94}, has been studied extensively by means of numerical simulation. This problem has broad application in a number of fields such as celestial mechanics, molecular dynamics, fluid dynamics (e.g., vortex methods), solid mechanics (elasticity, cracks, dislocation dynamics, etc.), and plasma physics. More broadly, the problem of computing $N^2$ interactions among $N$ points or $N$ variables appears in the boundary element methods, problems involving radial basis functions (e.g., interpolation, meshless algorithms, etc.), or in probability theory to describe dense covariance matrices (e.g., seismic imaging, linear stochastic inversion, kriging, Kalman filters, etc.). 

The essence of all of these simulations is to compute particle-particle interactions. Considering $N$ particles located at positions $\{ x_i \}$, the net contribution of these particles at some observation point $y$ is calculated by a sum of the form:
\begin{equation}
\label{FMM_SUM}
f(y)=\sum_{i=1}^{N} K(x_i,y) \sigma_i
\end{equation}
where $K$ is some $\mathbb{R}^3 \times \mathbb{R}^3 \rightarrow \mathbb{R}^3$ function called the kernel, and $\sigma_i$ is the intensity of the $i$'th particle field (particle mass in the celestial mechanic example).  The above formulation is essentially a matrix-vector multiplication, $A \boldsymbol{\sigma}$, where $[A]_{ij}=K(x_i,y_j)$. In general, the cost of this calculation is ${\cal O}(N^2)$. However, the fast multipole formulation introduced by Greengard and Rokhlin \cite{greengard87} allows a matrix-vector multiplication to be approximated with desired accuracy in ${\cal{O}}(N)$ time. This followed by several works to extend the algorithm for different kernels, analyze approximation, parallel implementation techniques, etc. We refer to some of the significant publications and our previous works \cite{greengard87,greengard97,fong09,darve2000a,darve2000b,darve11,yokota10,aparna10}.

An FMM matrix is a special class of ${\cal H}^2$-matrices with analytical low-rank off-diagonal blocks. ${\cal H}^2$-matrix itself is a subclass of a larger category of hierarchical matrices called ${\cal{H}}$. There are many fast linear algebra techniques for different classes of ${\cal{H}}$-matrices \cite{sivaram13}. As is explained in \S \ref{AFMM}, the main idea of fast linear algebra techniques for the hierarchical matrices is the low-rank approximation of interaction between well-separated clusters. In the context of celestial mechanics this idea translates to aggregating the effect of very far planets and computing their net gravitational force through a more efficient representation.

The adaptive FMM refers to the case where the particle distribution, and its corresponding hierarchical tree, is not uniform. The extension of the uniform FMM is well explained in \cite{carrier88, greengard99, zorin03}, followed by various aspects of its parallel implementation on different machines as discussed in \cite{zorin12, goude13}, and many other papers. 

Although the algorithm has been described several times in the past, previous papers have limited their analysis to very specific point distributions. We will detail this point shortly. The key point essentially is the manner in which points are distributed, in a non-uniform adaptive setting, as $N$ goes to infinity. In the uniform case, the issue of increasing $N$ presents no particular difficulty. We can simply increase the {\bf density} of points uniformly and study how accuracy and parameters in the FMM are adjusted as a function of $N$.

The non-uniform case however is more difficult. One essential point is describing the process of adding points so that $N \to \infty$. The adaptive test cases considered by most previous works fall broadly into the following categories:
\begin{enumerate}
\item A small number of subregions are picked (e.g., $n$ spheres) and points are progressively added to each subregion by distributing them with some smooth distribution (e.g., uniform, Gaussian, etc.) inside each region. Then the diameter and distance between regions are varied \cite{goude13,abeyratne13,zorin03}.
\item Manifolds are considered, that is surfaces or lines. Then, points are added on these manifolds again using a randomly uniform distribution \cite{bosch99, carrier88, greengard99,zorin03}.
\item Points are chosen such that they accumulate at some location, for example point $i$ is chosen as $x_i = 1/i^2$ \cite{carrier88, greengard99, huang09, singh93, singh95}.
\end{enumerate}
We note that complex cases have been considered such as in~\cite{greengard99, song97} but in those particular cases $N$ was fixed. 

All these cases represent only a small set of possible situations. There are many more ways to create non-uniform distribution of points. In this paper, we focus on the third case, in which points accumulate. However, we extend this situation to points that essentially accumulate at an infinite number of locations. This naturally leads to {\bf fractal sets.}

Fractal sets are encountered in several problems, notoriously when creating models of the universe~\cite{mandelbrot83,pietronero1987fractal,borgani1993multifractal,martinez1990universe,ribeiro1998fractals,james80,joyce2000}. Recently, antennas with fractal geometries~\cite{gianvittorio2002fractal,cohen1997fractal,vinoy2001hilbert,best2002resonant,cohen2000microstrip} have been studied. Their fractal shape allows a more compact design (taking advantage of the space-filling properties of fractal curves). Such contours are able to add more electrical length in less volume. In~\cite{parron2003method}, the method of moments and a variant of the fast multipole method~\cite{michielssen1996multilevel} is used to study iterated Sierpinski microstrip patch antennas (Sierpinski pre-fractal sets). Although the literature on applying the fast multipole method to fractal sets is limited, we mention~\cite{lang2005empirical}, which analyzed fast Gauss transforms and related fast methods for non-uniform data sets, described using their lacunarity~\cite{allain1991characterizing} (a concept related to the Hausdorff dimension for fractal sets).

Fractal sets, in addition to their physical importance, are interesting from the numerical point of view. Some previous works have reported a strong correlation between the computational cost of the FMM, and the dimension of points, based on observations for limited cases with an integer dimension (1, 2, or 3) \cite{greengard89, bosch99, gumerov03}. Fractal sets can have any dimension between 0 and 3 (not only integers), while filling a three-dimensional box. This more general definition of dimension is known as the fractal  dimension, box-counting dimension, or Hausdorff dimension~\cite{hutchinson81, dierk07}. Specifically in our numerical benchmarks, one of our main examples is a triple tensor product of the generalized Cantor sets, which provides all range of box-counting dimensions varying continuously from 0 to 3. Fractal sets, generally results in an adaptive FMM tree with leaf nodes sparsely distributed at different depths of the tree. Therefore, fractal sets with a continuously tunable dimension, are an interesting benchmark for the study of performance of the adaptive FMM codes. Furthermore, the performance of an adaptive FMM algorithm applied to a regular low dimensional benchmarks  (e.g., a 2D manifold in a 3D box) highly depends on the orientation of the set of points. In contrast, many fractal sets maintain their adaptive properties with rotation (i.e., they are more isotropic). This is the case for example when using tensor products of Cantor sets.

In \S \ref{AFMM}, we briefly introduce the adaptive FMM algorithm, and set of notations and definitions that are used in the rest of the paper. It is known that the adaptive FMM algorithm maintains the ${\cal{O}}(N)$ complexity irrespective of the point distribution \cite{nabors94, aluru1996greengard}. This requires a modification to the original FMM. In \S \ref{proof}, a new proof for the linear complexity of the adaptive FMM is introduced. We analyze the complexity of every interaction in the adaptive FMM step by step. This makes it apparent what modifications to the original FMM are required to ensure ${\cal{O}}(N)$ complexity for a general particle distribution. Moreover, the proof provides a heuristic for the reader about the effect of point distribution on the complexity of each part of the adaptive FMM.

Using a C++ adaptive FMM code (which can be downloaded from \url{stanford.edu/~hadip/AFMM.tar}) based on the black-box algorithm of \cite{fong09}, the aforementioned fractal point distributions are studied, along with a detailed counting of the number of floating point operations. In \S \ref{numerical}, the calculation begins with some standard cases (e.g., uniform, spiral, etc.), and evidence of how the point distribution changes the FMM cost and the optimal parameters is presented. The calculation is extended to general fractal sets afterwards.

When studying the adaptive FMM, the process of increasing $N$ should be considered rigorously. If we simply consider a sequence of points $x_i$, and define the set $S_N = \{x_i\}_{i=1,\ldots,N}$, the sequence of sets $S_N$ can only converge to a set with a countable number of points (by construction). However, we are interested to study the performance of the adaptive FMM for a sequence of sets that is converging to a fractal set (e.g., Cantor set). In \S \ref{fractal}, we introduce the concept of super octree, and how to define a sequence of octrees converging to a super octree. Each super octree is associated with a set of points (e.g., a fractal set). This provides the required mathematical framework to study a sequence of finite adaptive FMM problems (with $N \to \infty$), where point distributions have properties converging to the properties of some target infinite (and uncountable) set of points such as the Cantor set. In our numerical benchmarks we consider generalized Cantor sets that are constructed based on a recursive definition (find the code to generate points from \url{stanford.edu/~hadip/CANTOR.tar}). The points $x_i$ are generated by going through $k$ iterations of this recursive process. As $k \to \infty$, $N$ goes to infinity in a well-defined manner.

In the first part of \S \ref{optimization}, we propose a new strategy to build the adaptive tree. We focus on the criterion used to determine whether a cell needs to be further subdivided or not. The original bisection algorithm uses one threshold determined \emph{a priori} (e.g., \cite{greengard87, carrier88, zorin03}). Previous papers have also considered the median point, although that was mostly in the context of parallel implementations, which is also based on an \emph{a priori} threshold \cite{goude13}. We are proposing to use two independent parameters simultaneously: the maximum number of points per leaf cluster, and the maximum number of levels (depth) $\lmax$ of the tree. Therefore, we need to solve an optimization problem to find the optimal parameters.

Using two independent parameters turned out to be very important in order to reduce the computational cost. The main observation is that a leaf at level $\lmax$ generates mostly M2L operations, while leaves at levels lower than $\lmax$ also generate M2P and P2L operations. As we reduce the leaf particle threshold, we are replacing leaves at $\lmax$ by leaves at lower levels, which, in general, leads to fewer M2L and more M2P and P2L. These operators are not equivalent. M2L is pre-computed and can be accelerated using various techniques such as the fast Fourier transforms. In contrast, M2P and P2L are typically not precomputed (because this leads to a large memory storage) and are evaluated on the fly, as needed. They are consequently much more expensive. As our benchmark tests show, replacing M2L by M2P and P2L is disadvantageous. As a result, the particle threshold should be chosen very small so as to minimize the appearance of M2P and P2L. In particular, one outcome of this optimization is that the number of particles for clusters at $\lmax$ should be much larger than for leaf clusters at lower levels. This is in contrast with most adaptive FMM codes, that use a uniform threshold for all leaf nodes. These codes pay a high cost by generating too many expensive M2P and P2L operations.

In our numerical benchmarks, we introduce a near-optimal solution to this optimization problem, and demonstrate a better performance comparing to the conventional scheme. In the case where a very cheap M2L is possible (e.g., using fast Fourier transforms), our near-optimal solution brings about an even more dramatic computational cost reduction. One of our conclusions is that using the X-list and W-list (defined in \S \ref{AFMM}) in the adaptive FMM is not necessarily advantageous.

In section \ref{heuristic}, we study how parameters in the FMM such as the optimum total number of levels and the maximum number of points per leaf cells can be optimized as a function of the dimension of the set. We consider dimensions ranging continuously from 1 to 3, and show a strong dependence of these parameters on the dimension. Other details of the distribution appear to be less important. Our analysis is based both on mathematical bounds and estimates, as well as numerical benchmarks and investigations. Theoretical estimates for the optimal parameters can be found for uniform distributions~\cite{gumerov03}, whereas for a generic adaptive distribution not much is known. Most implementations, if not all, manually or heuristically tune parameters to get the optimal values of parameters (e.g., \cite{aparna10}). 

One can extend our analysis to design an automatic way of choosing the optimal parameters. For example, we can define local parameters for different parts of the domain (i.e., the maximum depth and threshold do not have to be globally defined). However, this is a separate work that is not addressed here.




\section {Adaptive Fast Multipole Method} \label{AFMM}

We briefly introduce the adaptive FMM in this section, and refer the reader to \cite{zorin03} for a detailed explanation. The fundamental idea in FMM is to avoid detailed computation among particles that are far from each other, and approximate these with some low rank operator. This reduces complexity from ${\cal{O}}(N^2)$ to ${\cal{O}}(N)$. Hence, the main idea is to cluster points. In order to keep track of different clusters of particles we can use a tree data structure. In this section we explain and illustrate the scheme in a 2D domain for simplicity. Also, hereafter we assume observation points are the same as particle locations (i.e., $y$ is one of the $x_i$'s in Eq.~\eqref{FMM_SUM}). 

Consider a set of $N$ particles in a square domain, $[0, 1]^2 \subset \mathbb{R}^2$. To build our hierarchical data structure of clustered particles, we start from the original box (which is a $1\times1$ square in our 2D example), and subdivide it into four identical sub-domains. We keep subdividing new boxes as long as the number of particles per box is greater than some \emph{threshold} $t$, determined a priori. For instance, in Figure \ref{domain} the original box is subdivided up to 5 levels, assuming subdividing threshold is $t=8$. The largest box corresponds to the root of the tree, which has four children associated with four sub-boxes. Each of these four children may have up to four new children, and so forth. Therefore, each box in the geometrical (physical) domain is corresponding to a node in the tree (this is defined below). Note that in a $d$-dimensional computation, each node has $2^d$ children, therefore, a $2^d$-tree is the proper substitution for the quad-tree of the 2D case (e.g., octree if $d=3$).

\begin{definition} (physical interval associated to a tree node) As described above, the physical cube associated to a node $M$ in the octree is defined recursively as a $1/{2^d}$ cube of the physical cube of $M$'s parent in the tree. We denote this cube by ${\cal{I}}(M) \subseteq [0, 1]^d$, which is a \emph{closed} subset of $\mathbb{R}^d$ . $d \in \{1,2,3\}$ is the topological dimension of the problem.
\label{def:I_M}
\end{definition}

\begin{figure}[!htbp] \centering
\includegraphics[width=0.50\textwidth]{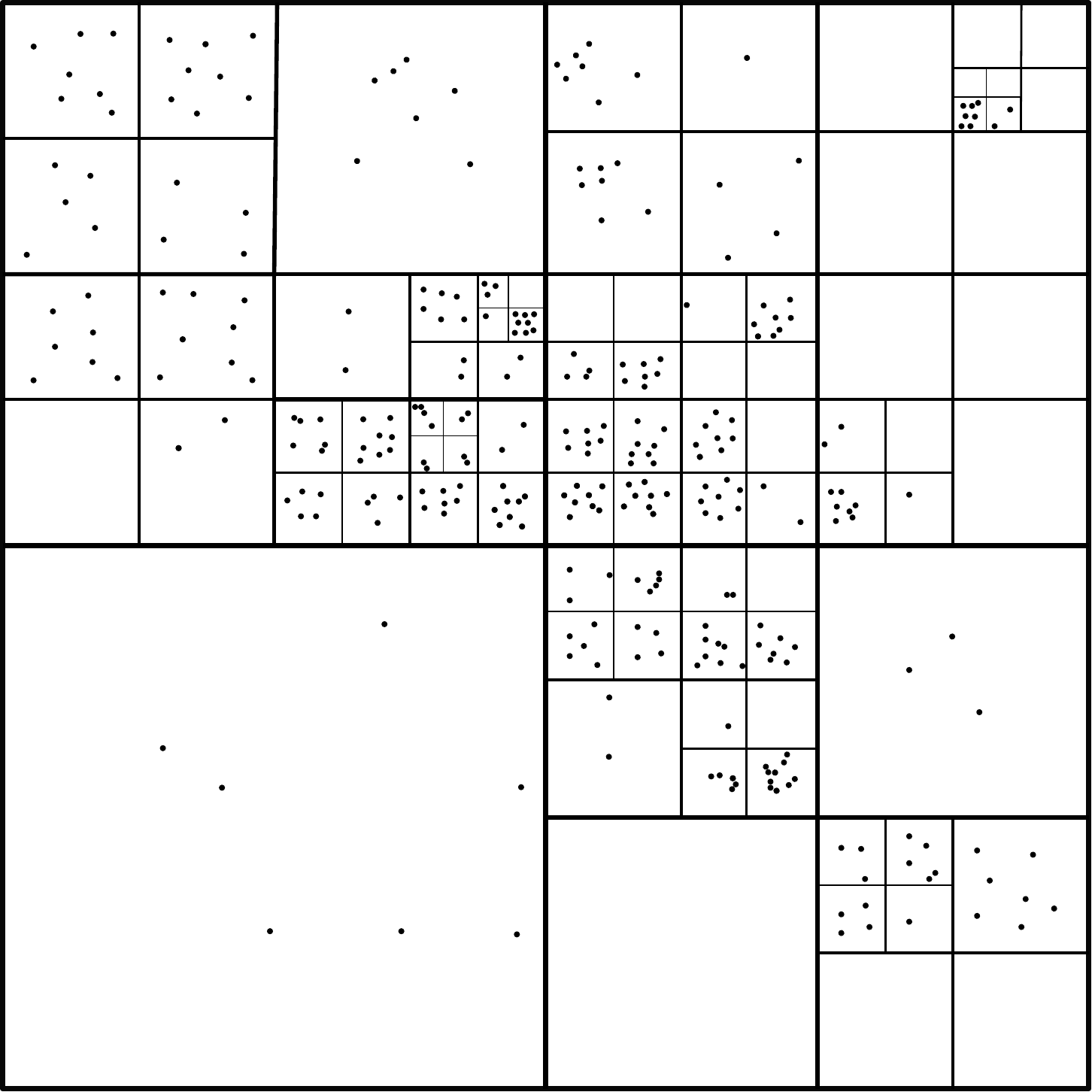}
\caption{A 2D non-uniform points distribution domain. The threshold for subdivision is $t=8$.}
\label{domain} 
\end{figure}

Now, based on the FMM's idea, we have to distribute the total computation amongst different nodes of the tree (i.e., each box in the domain). We will use the following definitions for the hierarchical tree similar to \cite{zorin03}.

\begin{definition} (node relations in the tree)
\begin{itemize}
\item Node $A$ is \emph{adjacent} to node $B$ if ${\cal{I}}(A) \cap {\cal{I}}(B) \ne \emptyset$.
\item Node $A$ is a \emph{parent} of node $B$ if $B$ is child of $A$ in the tree. We denote it by $A={\cal{P}}(B)$.
\item \emph{Depth} of a node $A$ in the tree is defined recursively as $\delta(A) = \delta( {\cal{P}}(A) ) +1$. Depth of the root is 0.
\item Node $A$ is a \emph{colleague} of node $B$ if they are adjacent and $\delta(A) = \delta(B)$.
\item Node $A$ is a \emph{leaf} if it has no children. It corresponds to boxes with at most $t$ particles.
\end{itemize}
\end {definition}

In the FMM, two sets of coefficients are defined for each node: multipole and local coefficients. Local coefficients consist of the contribution of all far particles to the particles within the node of interest. Multipole coefficients include the contribution of particles inside the node of interest to any far particle. The number of multipole and local coefficients assumed to be constant irrespective of the depth of the node in the tree. Considering a leaf cluster of particles, say C, we can divide its surrounding area into three regions: \emph{vicinity}, \emph{separated}, and \emph{far} (Figure \ref{regions}). In the following these regions are defined.

\begin{figure}[!htbp] \centering
\hspace{20mm}
\includegraphics[width=0.40\textwidth]{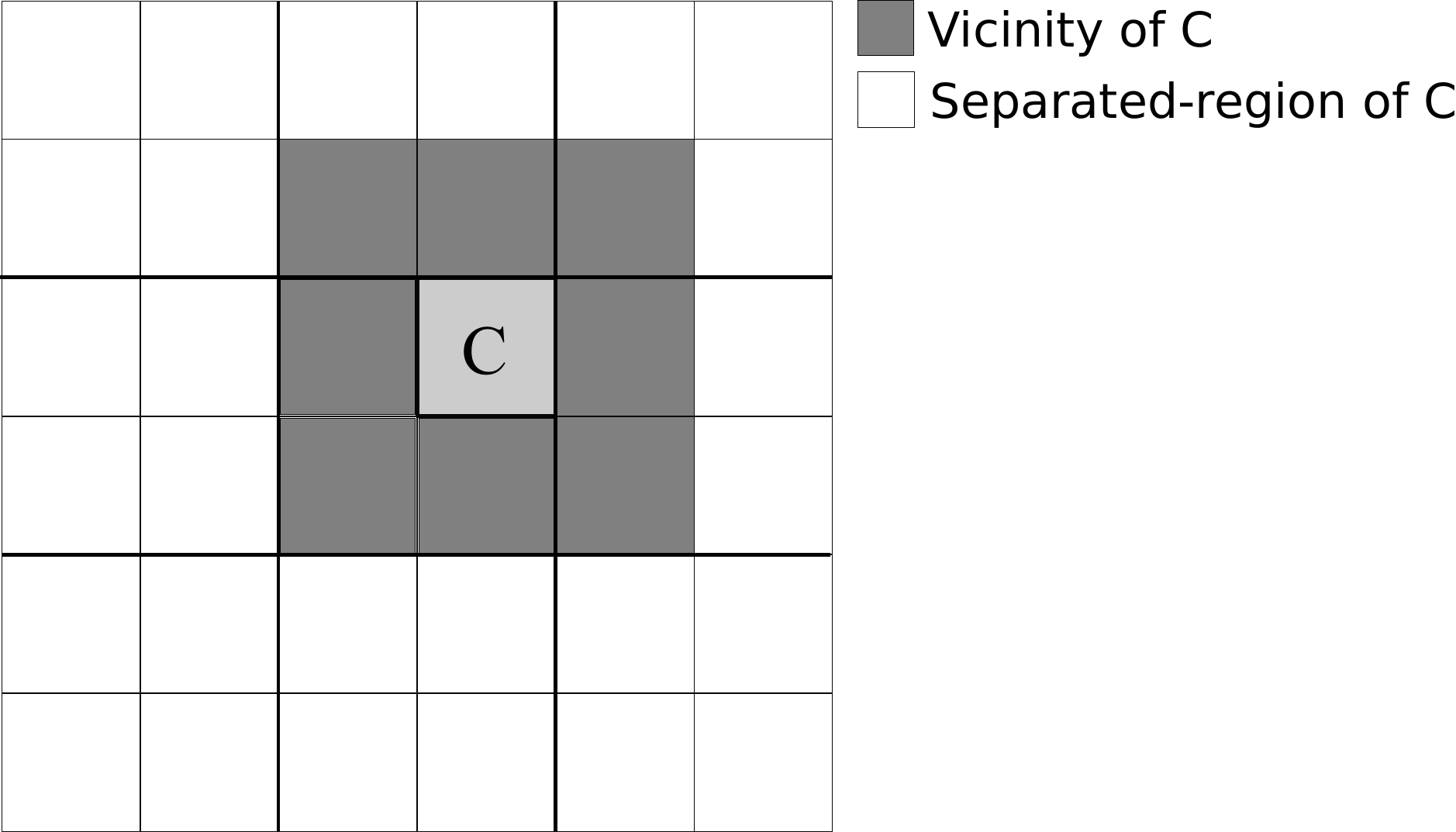}
\caption{Vicinity and separated regions around a node in the FMM tree.}
\label{regions} 
\end{figure}

\begin{definition} (regions around a cluster of particles)
\begin{itemize}
\item The \emph{vicinity} of a node C is defined as $\cal{V}($C$) = \cup_{\alpha} {\cal I}(\alpha)$, where $\alpha$ is a potential colleague of C, whether it exists in the tree or not (hence, the union is over 8 cubes in 2D, and 26 cubes in 3D).
\item Node $\alpha$ is \emph{separated} from C if $\alpha \cap \cal{V}(\cal{P}($C$)) \ne \emptyset$ and  $\alpha \cap \cal{V}($C$) = \emptyset$.
\item Node $\alpha$ is \emph{well-separated} from C, if it is separated from C and $\delta(\alpha) \le \delta(\mbox{C})$. Therefore, If $\alpha$ is \emph{well-separated} from C, then C is certainly separated from $\alpha$. Note that in the context of the adaptive FMM, being well-separated is a one-way relation.
\item Node $\alpha$ is \emph{far} from C, if it is neither in vicinity of C nor separated from C.
\end{itemize}
\end{definition}

The interaction of particles in a cluster C with the particles within its \emph{adjacent} area is presumably not low-rank, and a direct particle-to-particle (P2P) calculation is required. The interaction with particles in the \emph{separated} region, however, can be approximated with a low-rank operator, M2L (multipole to local). The particles lying in the \emph{far} region, by construction, are within the \emph{separated} region of C's ancestors. Therefore, C inherits the far-field contribution from its ancestors through the L2L (local to local) operator, and transfers the effect of its local particles through the the M2M (multipole to multipole) operator to its ancestors. In order to implement this hierarchical data structure to transfer information among nodes, we introduce the following lists for every node C of the tree:
\begin{definition} (interaction lists of a node C in the tree)
\begin{itemize}
\item U-List (only defined for leaf nodes): $\alpha \in$ U-List(C) if and only if $\alpha$ is a leaf node adjacent to C.
\item V-List: $\alpha \in$ V-List(C) if and only if both $\alpha$ and C are well-separated from each other.
\item W-List (only defined for leaf nodes): $\alpha \in$ W-List(C) if and only if C is well-separated from $\alpha$, but $\alpha$ is not well-separated from C.
\item X-List: $\alpha \in$ X-List(C) if and only if C $\in$ W-List($\alpha$).
\end{itemize}
\end{definition}
\begin{remark}
\item C $\in$ U-List(C).
\item $\alpha \in$ U-List(C) implies that C $\in$ U-List($\alpha$).
\item $\alpha \in$ V-List(C) implies that $\delta(\alpha) = \delta(\mbox{C})$. Also it implies that C $\in$ V-List($\alpha$).
\item $\alpha \in$ X-List(C) implies that $\alpha$ is a leaf (since it has a W-List).
\end{remark}

In Figure~\ref{lists}, the aforementioned lists are depicted for a leaf node of Figure~\ref{domain} as an example. Essentially, the U-List is designed to include all the nodes within the adjacent region. Therefore, a direct interaction with nodes in the U-List is required. The concept of separation provides the necessary condition for a low-rank approximation. V-List, W-List, and X-List include nodes with low-rank interactions. Consider two nodes $\alpha$ and $\beta$. When $\beta$ is \emph{separated} from $\alpha$, the effect of source points in node $\alpha$ to the observation points in node $\beta$ is approximated by $\alpha$'s multipole coefficients. Also, the effect of source points of $\beta$ to the observation points of $\alpha$ is added to the $\alpha$'s local coefficients. 

According to the above explanation, the adaptive FMM algorithm uses the following 8 sub-algorithms to compute the effect of all source points on the all observation points.
\begin{enumerate}
\item P2M: Aggregate particle information into the multipole coefficients. This is applied to all leaf nodes.
\item M2M: Compute the contribution of multipoles of a node on the multipoles of its parent (use post-order tree traversal).
\item M2L: Low-rank interaction between multipoles and locals of two separated nodes (use V-List).
\item P2L: Low rank interaction between particles and local coefficients (use X-List).
\item L2L: Compute the contribution of locals of a node on the locals of its children (use pre-order tree traversal).
\item M2P: Low-rank interaction between multipoles and particles (use W-List).
\item P2P: Direct interaction between particles (use U-List).
\item L2P: Interpolate local coefficients to particle location (applied to all leaf nodes).
\end{enumerate}	

\begin{figure}[!htbp] \centering
\includegraphics[width=0.50\textwidth]{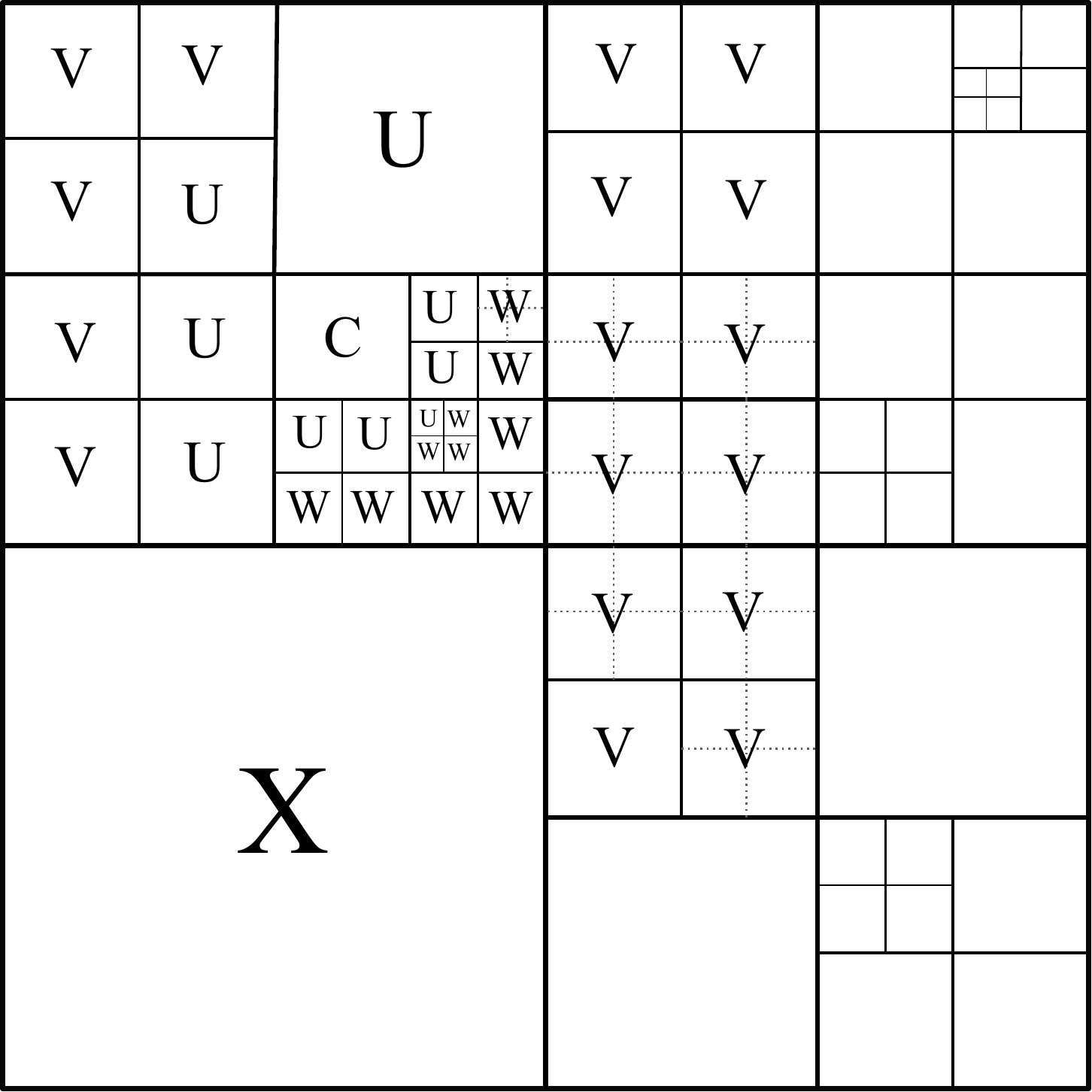}
\caption{Example of U, V, W, and X lists for a leaf node C in the example of Figure \ref{domain}.}
\label{lists} 
\end{figure}

\section {Proof of linear complexity} \label{proof}

In the previous section, we discussed how the complete algorithm consists of eight sub-algorithms: P2M, M2M, M2L, L2L, L2P, P2L, M2P, and P2P. 
Here, we show that the adaptive fast multipole method has an ${\cal{O}}(N)$ operation count, where $N$ is the number of particles. We assume observation and source points are the same, and are located in $\mathbb{R}^3$. Also we assume that every low-rank approximation has a constant rank, independent of the depth of the nodes.
 
Note that for the adaptive case, the octree is not necessarily full (i.e., each node does not necessarily have 8 children filled with particles). Therefore, the depth of the octree for the adaptive case is not bounded by ${\cal O} (\log N)$, compared with the uniform case. Nabors et.\ al.\ have proposed a proof for linear complexity of the adaptive FMM~\cite{nabors94}. \cite{nabors94} does not use the X and W lists. Instead, clusters interact only with clusters at the same level (with the exception of non-divided nodes, see~\cite{nabors94} p.~721). \cite{aluru1996greengard} also provides a proof for linear complexity of the FMM independent of distribution of points. However, they do not discuss the role of the subdividing threshold, and the adaptive interaction lists X and W. In ~\cite{sevilgen2000provably} the parallel implementation of their distribution independent algorithm is discussed. We present a new proof of the linear complexity with the usual set of interaction lists U, V, W, and X. This is done for each sub-algorithm separately. During the proof we introduce some modifications to the basic algorithm, which is necessary to get an ${\cal O}(N)$ complexity. Here, we assume a generic adaptive octree consists of $N$ particles with a given subdividing threshold $t$, and prove the linear complexity step-by-step.

\begin{figure}[!htbp] \centering
\includegraphics[width=0.40\textwidth]{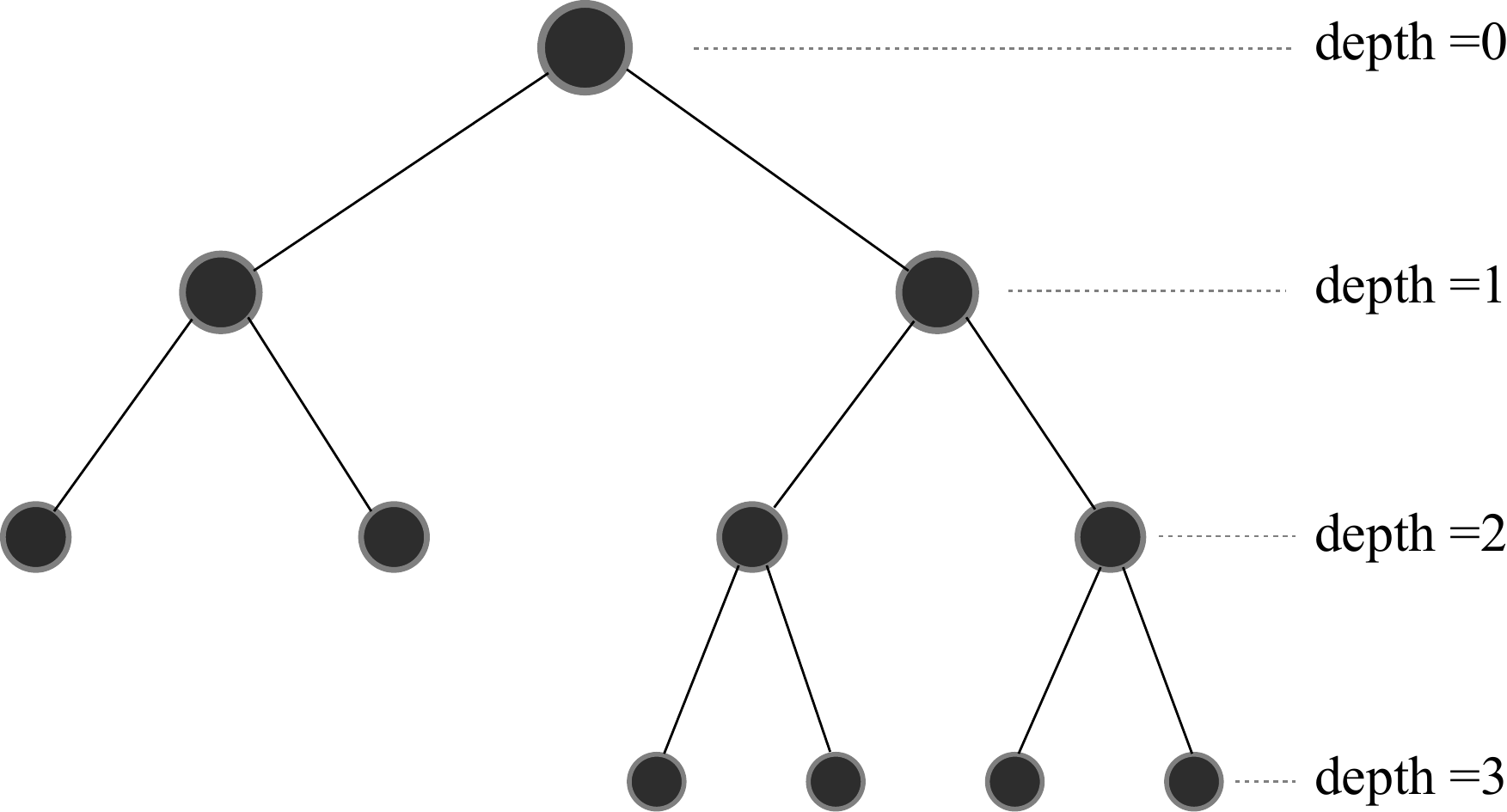}
\caption{Adaptive binary tree (for the 1D FMM). In the above tree, assuming that the subdividing threshold is 2, there is a total of 10 observation/source points. }
\label{adaptive_tree}
\end{figure}

\begin{lemma}
\emph{(P2M \& L2P)}
In any adaptive FMM tree with $N$ particles, there are exactly $N$ P2M, and $N$ L2P operations.
\end{lemma}
\begin{proof}
The number of P2M (and similarly L2P) operations is independent of the particles distribution and subdividing threshold. Only the number of multipole/local coefficients per node (which depends on the algebraic low-rank approximation, e.g., Chebyshev polynomial interpolation) is relevant. Hence, there are a total of $2N$ P2M and L2P operations regardless of particles distribution and the subdividing threshold. \qed
\end{proof}

\begin{lemma}
\emph{(P2P)}
In any adaptive FMM octree with $N$ particles, the number of P2P operations is $\mathcal{O}(N)$.
\end{lemma}
\begin{proof}
Define a directed graph $G(V,E)$ as follows. There is a one-to-one mapping between the leaf nodes in the octree and the vertices in $V$. Then, there is an edge $e_{ij} \in E$ from vertex $v_i$ to vertex $v_j$ if and only if, node $i$ is adjacent to node $j$ in the octree, and the depth of node $j$ is less than or equal to the depth of node $i$.

In the octree, each node has at most $3^3-1=26$ adjacent nodes at depth less than or equal to its own depth. Therefore, in the graph $G$, each vertex has at most 26 outgoing edges. This means that there are at most $26N_l$ edges in the graph, where $N_l$ is the number of leaf nodes in the octree, and clearly, $N_l \le N$ (in an average sense $N_l \sim N/t$). A P2P operation between two nodes takes ${\cal{O}}(t^2)$ CPU cycles. Hence, the total number of CPU cycles associated to the P2P part of the algorithm is ${\cal{O}}(N_l t^2)$, which is ${\cal{O}}(N)$ for a given threshold $t$. \qed
\end{proof}

For the rest of our complexity analysis, we need to define the concept of \emph{extended tree}. This helps us to find appropriate upper bounds, as will be explained.
\begin{definition} (extended tree) The \emph{extended tree} is obtained from a basic tree (Figure \ref{adaptive_tree}) by performing the following modifications. 

\begin {itemize}
\item Denote the depth of the tree by $\lmax$.
\item For every non-leaf node consider all of its 8 children. Some of them are occupied with particles, and some of them are empty. We call the former an \emph{occupied node}, and the latter an \emph{empty node}. Add empty nodes to the tree.
\item Consider all leaf nodes with depth less than $\lmax$. They may be occupied or empty. Subdivide leaf nodes to 8 children, and continue this process until reaching the deepest level of the tree (i.e., $\lmax$).
\label{step:unif}
\end{itemize}

Doing the above modifications, we end up with a tree whose leaves are all in the deepest level. Leaves are either occupied or empty. For instance, the binary tree with 10 particles shown in Figure \ref{adaptive_tree} transforms to the extended binary tree in Figure \ref{extended_tree}.
\end{definition}

\begin{figure}[!htbp] \centering
\includegraphics[width=0.5\textwidth]{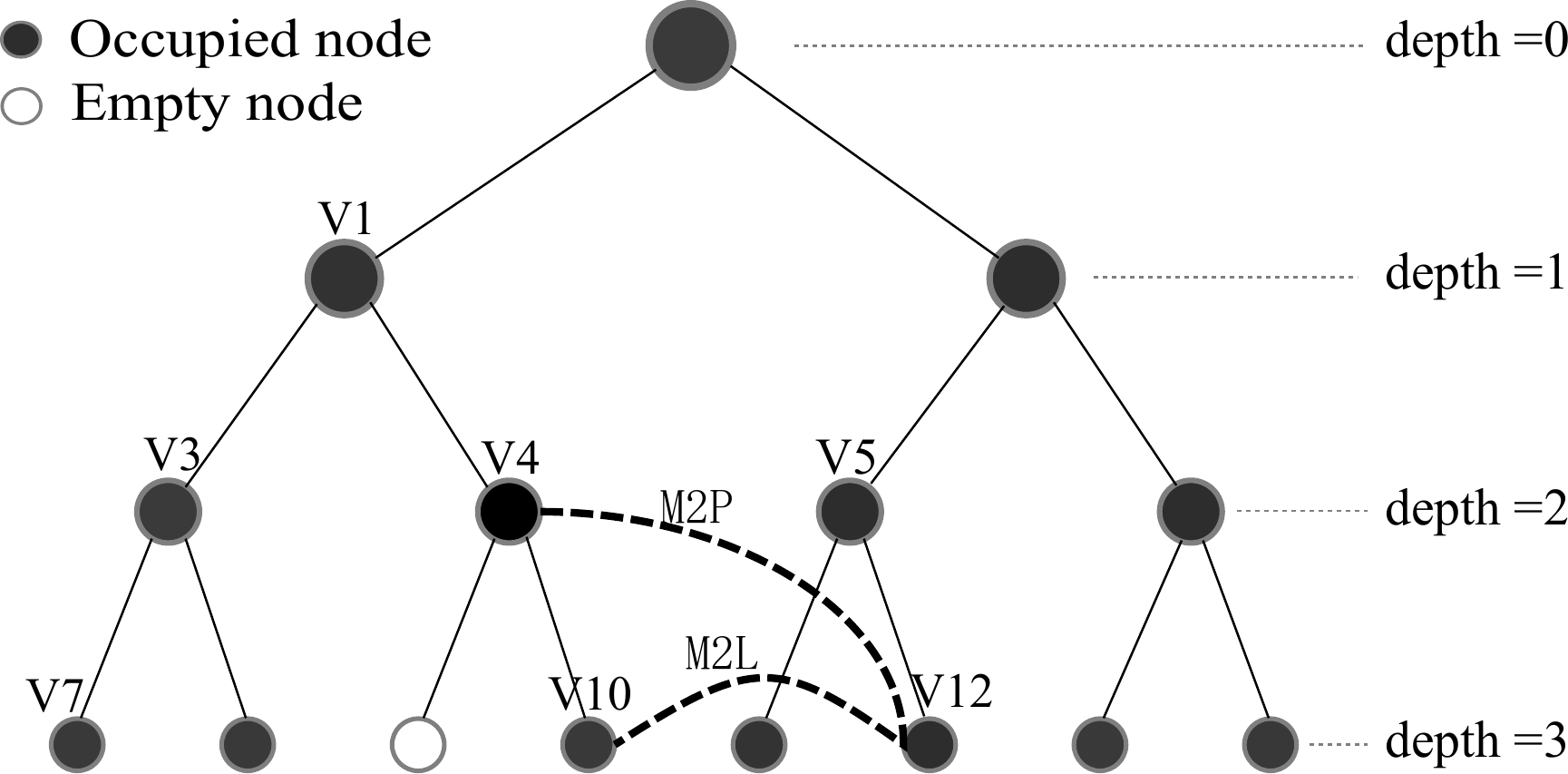}
\caption{Extended adaptive binary tree corresponding to Figure \ref{adaptive_tree}. By definition, in the extended tree all leaves lie in the deepest level. An example of an M2P operation transformed to a new M2L operation is depicted.}
\label{extended_tree}
\end{figure}

\begin{lemma}
\emph{(M2P \& P2L)}
In the extended adaptive FMM tree, there are no M2P and P2L operations.
\end{lemma}
\begin{proof}
By definition, in the extended octree all leaves lie in the deepest level, so no M2P and P2L is possible. \qed

Note that as we transform the basic octree to the extended octree, each M2P operation transforms to new M2L operations (at lease one new M2L). For instance, in Figure \ref{extended_tree}, the M2P operation between nodes $V4$ and $V12$, which exists in the basic tree of Figure \ref{adaptive_tree}, transforms to an M2L between nodes $V10$ and $V12$. Basically, in the tree extension algorithm, all M2P and P2L (which is the dual of M2P) operations transform to new M2L operations, such that the number of M2L operations in the extended tree is not less than the total number of M2Ls, M2Ps, and P2Ls in the original octree. Therefore, it is sufficient to show that number of $M2L$ operations in the extended octree is linear with $N$. Note that the number of CPU cycles associated to each M2P and P2L operation is only function of $t$ and $r$ (the low-rank approximation), and the flops of each M2L is only function of $r$. Hence, assuming constant $t$ and $r$, it is sufficient to count the total number of operations to show the linear complexity with respect to $N$.
\end{proof}

\begin{definition} (divided node, parent, and children)
\begin{itemize}
\item A \emph{divided node} in the tree is an occupied node that has at least two occupied children. Similarly, a \emph{singleton node} is defined as an occupied node that has exactly one child. Note that leaves and empty nodes are neither divided, nor singleton.
\item The \emph{divided parent} of a node $\beta$ is defined as its deepest ancestor who is a divided node or the root. For instance, in Figure \ref{extended_tree} $V1$ is the divided parent of $V10$. 
\item The \emph{divided children} of a given node $\beta$ is defined as its least deep (shallowest) descendant which is a divided node or a leaf.
\end{itemize}
\end{definition}

\begin{lemma}
(number of divided nodes)
\label{lemma1}
For a generic adaptive extended octree $T$, the number of divided nodes is less than $N$.
\end{lemma}
\begin{proof}
We can show a stronger result by induction on the depth of the extended octree,  $l$: 
\begin{equation}
\label{lemma_ineq}
\sum_{v_i \in \mbox{ occupied non-leaf nodes}} (d(v_i)-1) =N_l-1
\end{equation}
where, $d(v_i)$ is the number of occupied children $v_i$ has, and $N_l$ is the number of occupied leaves.

The case $l=0$ is trivial.
Consider an extended octree with depth $l$. Note that all leaves in the extended octree lie in the deepest level. Remove the last (deepest) level of the tree. Call the remaining octree $T'$, which is an extended octree with depth $l-1$, and $N'_l$ occupied leaves. From the induction hypothesis the equality holds for $T'$. Now, subdivide all leaves in $T'$ to rebuild $T$. Empty leaves do not contribute to the above summation for both $T$ and $T'$. Consider an occupied leaf $v'_i$ in $T'$ that is subdivided to $d(v'_i)$ occupied leaves in $T$. Essentially, going from $T'$ to $T$, we loose one leaf ($v'_i$), and add $d(v'_i)$ new leaves. So the total number of leaves is increased by $\sum (d(v'_i)-1)= N_l- N'_l$. Add $\sum (d(v'_i)-1)$ to the left hand side, and $N_l- N'_l$ to the right hand side of the equality for $T'$ to obtain the equality for $T$.

Now, note that $d(v_i)-1$ is non-zero for all divided nodes. So:
\begin{equation}
\text{Number of divided nodes} \le \sum_{v_i \in \text{ occupied non-leaf nodes}} (d(v_i)-1)=N_l-1 \le N-1
\end{equation}
The last inequality holds, since each occupied leaf consists of at least one particle. \qed
\end{proof}

\begin{lemma}
\emph{(M2L)}
In any adaptive FMM tree with $N$ particles, the number of M2L operations is $\mathcal{O}(N)$.
\end{lemma}
\begin{proof}
To show linear complexity, we build the following bipartite graph. One set of nodes is the set of all M2L operations. The other set is the set of all divided tree nodes. We construct edges from each M2L operation to one or two divided tree nodes. We showed previously (Lemma~\ref{lemma1}) that the number of divided nodes is $\mathcal{O}(N)$ (in fact, it is $\mathcal{O}(N_l)$). We will then bound the maximum number of edges incoming on each divided node. This proves that the maximum number of M2L operations is $O(N)$.

Consider two generic nodes $\alpha$ and $\beta$. In a tree, there is a single path that connects $\alpha$ to $\beta$ and that visits a node at most once.
Edges in our bipartite graph connect an M2L operation between two generic nodes $\alpha$ and $\beta$ to a node $S$ such that $S$ is the {\bf deepest divided node in the path from $\alpha$ to $\beta$} including $\alpha$ and $\beta$ themselves. Since the path is unique, there are at most two nodes $S$ that are connected to an M2L operation.

Consider for example Figure \ref{extended_tree}. The M2L operation between nodes $V10$ and $V12$ is connected to node $V12$, and the M2L operation between nodes $V10$ and $V7$ is connected to node $V3$. In some cases, the divided node $S$ is not unique, e.g., the M2L between nodes $V3$ and $V5$ in Figure \ref{extended_tree}. In that case, the M2L is connected to both possible divided nodes.

Now, we show that the maximum number of edges incoming on some divided node $S$ is bounded by a constant.

If the M2L operation between nodes $\alpha$ and $\beta$ is mapped to $S$, by definition, at least one of $\alpha$ and $\beta$ should be a descendant of $S$, or $S$ itself.

If one of $\alpha$ or $\beta$ is $S$, then there are at most $6^3-3^3=189$ M2L operations of this type mapped to $S$, since this is the maximum size for the V-List of $S$.

Consider now all M2L operations connected to $S$ such that neither $\alpha$ nor $\beta$ is $S$. Without loss of generality, assume that $\alpha$ is a descendant of $S$ (if not true, then $\beta$ must be a descendant). Node $\alpha$ must belong to one of the children of $S$, $K_1$, $K_2$, \ldots, or $K_8$.

For each $K_i$, we now consider the maximum number of M2L such that $\alpha$ is a descendant of $K_i$ or $K_i$ itself. Given $K_i$, there are at most 215 not-far neighbors,\footnote{A not-far neighbor of a node C is a node with the same depth, which is either its colleague or a child of its parent's colleagues. This is basically the union of the colleagues and the V-List nodes (Figure \ref{regions}). Therefore, each node has at most $6^3-1=215$ not-far neighbors. Note that C or any of its descendants may only have M2L interactions with one of C's not-far neighbors and their descendants.} $N_{i,1}$, $N_{i,2}$, \ldots, $N_{i,215}$. At most 189 of them are well-separated from $K_i$, and at most 26 are colleagues of $K_i$ (they cannot be $K_i$ itself because $\beta$ cannot be in $K_i$ at this point). The node $\beta$ must be a descendant of one of these not far neighbors (this is true by definition of $S$).

We now show that for any pair $(K_i,N_{i,j})$, there is at most one set $\{ \alpha, \beta \}$ such that: $\alpha$ is $K_i$ or its descendant, $\beta$ is $N_{i,j}$ or its descendant, and there is an M2L between $\alpha$ and $\beta$ that is mapped to $S$. If we prove this, we will have established that there cannot be more than $189 + 8 \times 215 = 1909$ incoming edges on $S$.

For this final point, there are now only 2 cases. If $N_{i,j}$ is well-separated from $K_i$ (i.e., $N_{i,j}$ is in the V-List of $K_i$) then we must have $\alpha=K_i$ and $\beta=N_{i,j}$. In this case, we are guaranteed that there is no more M2L between descendants of $\alpha$ and $\beta$. The set $\{ \alpha, \beta \}$, if it exists, must be unique.

If $N_{i,j}$ is a colleague of $K_i$ and the M2L operation $\{ \alpha, \beta \}$ exists, both $N_{i,j}$ and $K_i$ must be singleton nodes (otherwise, the M2L would not be mapped to $S$). Because of the definition of $S$, we further have that the two subtrees starting at $K_i$ and $N_{i,j}$ must be branches of singleton nodes, at least until $\alpha$ and $\beta$ are reached. Moreover, no other M2L interaction can exist along these two branches. Essentially to find $\alpha$ and $\beta$, we simply go down the tree starting at $K_i$ and $N_{i,j}$ until we find two nodes that are well-separated. These must be $\alpha$ and $\beta$ and, consequently, the set $\{ \alpha, \beta \}$ is unique. \qed
\end{proof}

So far, we have discussed the complexity of six out of eight operations of the adaptive FMM. However, without the following modification, the basic adaptive FMM is not ${\cal O} (N)$. In fact, for a given $N$, it can be arbitrarily large compared with $N$. Consider a very long branch of singleton nodes ending at a leaf (see Example \ref{crazy_dist}). This branch of tree may be generated only due to one particle, and there is no limit on its depth. Such cases also have been considered in ~\cite{aluru1996greengard}. Consequently, performing the basic M2M (and similarly L2L) operator on this branch, we could add infinitely large cost to the algorithm. This is why we have to modify the algorithm. This modification is similar to the modification introduced in ~\cite{aluru1996greengard}.

\begin{definition} (modified adaptive FMM)
In the \emph{modified adaptive FMM algorithm} we do not consider multipole and local coefficients for singleton nodes. Specifically, for M2M, we pass the multipole coefficients of a non-singleton node \emph{directly} to its divided parent (as opposed to passing the multipole coefficients step-by-step through all the singleton ancestors). Similarly for L2L, we pass the local coefficients of a divided node directly to its divided children. For example, in the binary tree of Figure \ref{extended_tree}, the multipoles of node $V10$ directly contribute to its divided parent $V1$.
\end{definition}

\begin{remark}
M2L operations for singleton nodes, which in the modified algorithm do not have multipole and local coefficients anymore, are being taken care by their divided children. Moreover, going from the basic algorithm to this modified algorithm, the number of M2L operations does not change. Therefore, the proof for the M2L still holds.
\end{remark}

\begin{lemma}
\emph{(M2M \& L2L)}
In the modified adaptive FMM algorithm with $N$ particles, the number of M2M and L2L operations is $\mathcal{O}(N)$.
\end{lemma}
\begin{proof}
As explained, in the modified adaptive FMM, M2M and L2L operate between a divided node and one of its non-singleton children. From lemma~\ref{lemma1},	 we have a bound on these pairs:
\begin{gather*}
\mbox{Number of divided nodes} \le \sum_{v_i \in \text{ occupied non-leaf nodes}} (d(v_i)-1)=N_l-1 < N \\
\Rightarrow \sum_{v_i \in \text{ divided nodes}} d(v_i) = \left( \sum_{v_i \in \text{ divided nodes}} (d(v_i)-1) \right) + \text{Number of divided nodes} < 2N
\end{gather*}
This shows that the number of pairs of nodes that require M2M and L2L operators is bounded by $2N$ (in fact, it is bounded by $2N_l$). \qed
\end{proof}

\section {Numerical consideration} \label{numerical}

In \S \ref{proof}, we established an upper bound for each sub-algorithm of the adaptive FMM. In this section, we compute the CPU cylce count of the algorithm, and then show some numerical results to verify the linear complexity. We chose a low-rank approximation based on the Chebyshev polynomial interpolation introduced in~\cite{fong09}.

\subsection{CPU cycle count}

Using the scheme described in \cite{fong09}, we assumed $r$ Chebyshev points in each direction. Hence, for a 3D calculation, each node has $r^3$ multipole and local coefficients. Let's assume each kernel evaluation takes at most $k$ cycles. The aforementioned sub-algorithm operators lead to many small matrix-vector multiplications. We assume that the pre-computation of M2L, M2M, and L2L matrices is done before the actual calculation (the entries of these matrices are independent of the particle positions). 
In the current black-box adaptive FMM implementation, one application of each operation to an octree node has the following cycle count:\\
\begin{tabular}{ll}
P2M & $C_1 + C_2 r^4$\\
M2M & $C_3 + C_4 r^6$\\
M2L & $C_5 + C_6 r^6$\\
P2L & $C_7 + C_8 k r^3$\\
L2L & $C_9 + C_{10} r^6$\\
L2P & $C_{11} + C_{12} r^4$\\
M2P & $C_{13} + C_{14} k r^3$\\
P2P & $C_{15}+ C_{16} k$\\ 
\end{tabular}\\
where, in the above expressions all $C_i$'s are constants that depend on machine architecture, particle distribution, subdividing threshold, etc. Note that this cost can be reduced using different methods. For example, fast Fourier transforms can be used to reduce the M2L operator to $O(r^3 \log r)$.

\subsection{Numerical results}

In this section, numerical results for a black-box adaptive FMM calculation for various particle distributions are presented. Here, we have used $r=4$ Chebyshev points in each direction (i.e., a total of 64 points) for the low-rank approximation. In Figure \ref{point_distribution}, the particle clustering for the uniform and spiral distribution is depicted. Note that the corresponding tree for the spiral case is not fully adaptive in the sense that it has only one concentration point as $N \rightarrow \infty$. In Figure \ref{cost}, the total FMM cost (in giga cycles) is plotted versus the number of points. It is clear that the adaptive FMM has linear complexity, as opposed to the direct matrix-vector multiplication that has ${\cal{O}}(N^2)$ complexity. We have assumed each kernel evaluation takes 19 cycles, each division or square root takes 4 cycles, and each addition, multiplication, subtraction, or branching is assumed to take 1 cycle. We count cycles as the code is running. Note that in this analysis we are not taking to the account the cost of memory access. Memory access optimization is highly architecture dependent.

\begin{figure}[!htbp] \centering
\includegraphics[width=0.4\textwidth]{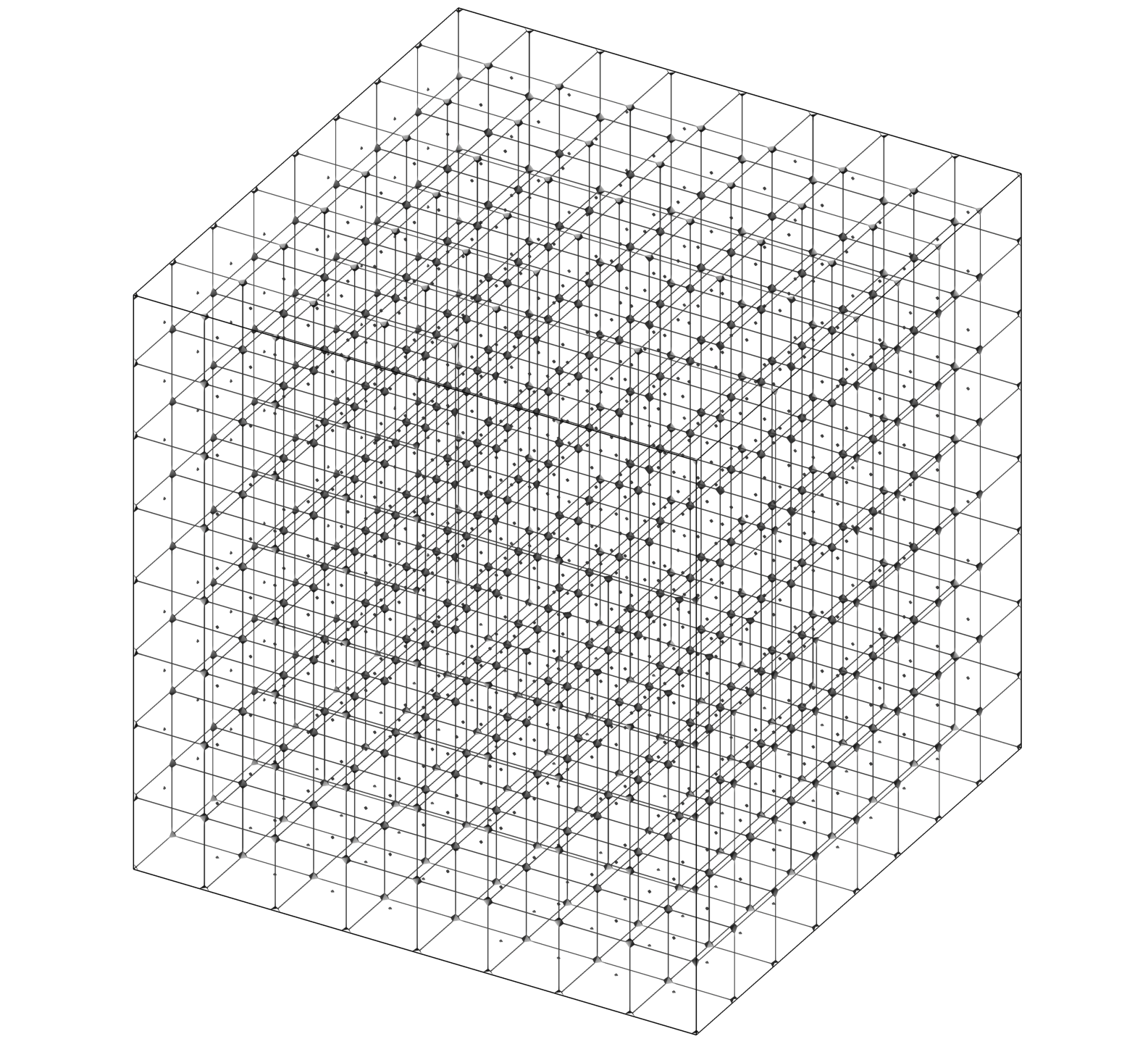}
\includegraphics[width=0.4\textwidth]{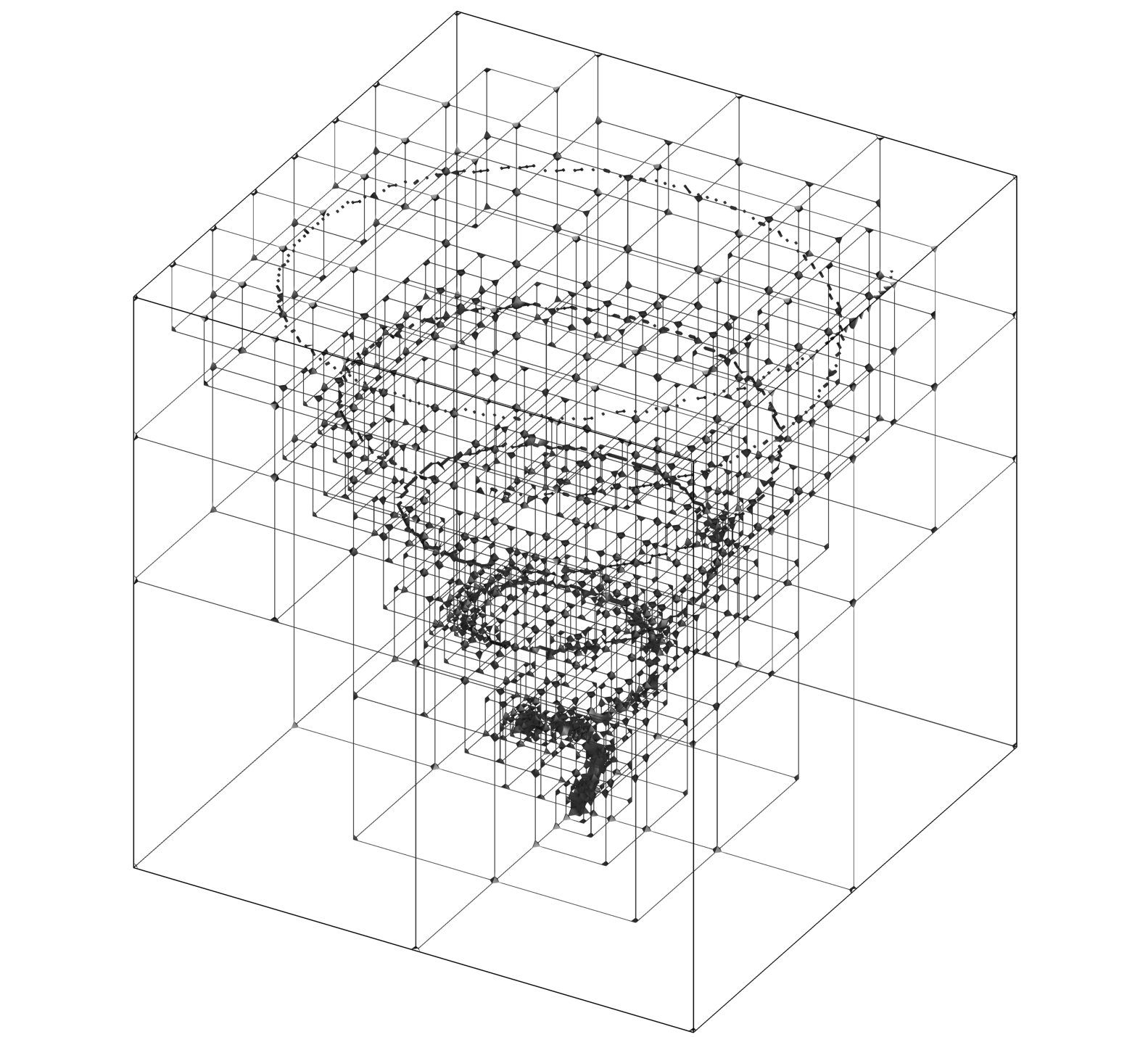}
\caption{Three dimensional adaptive FMM clusters of 1,000 particles with a subdividing threshold $t=10$, for two different point distributions. Left: uniform; Right: spiral.}
\label{point_distribution}
\end{figure}

\begin{figure}[!htbp] \centering
\includegraphics[width=0.45\textwidth]{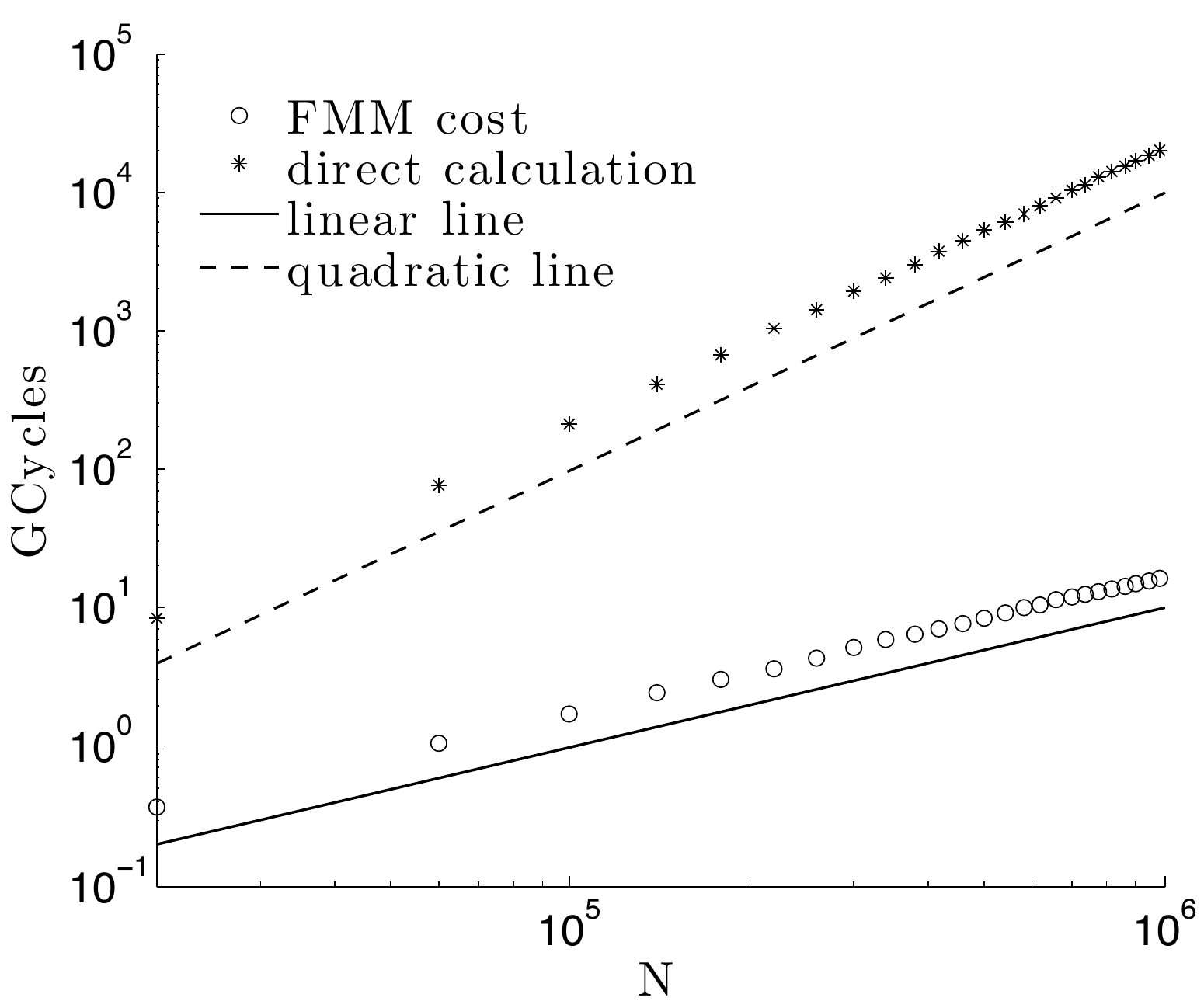}
\caption{Adaptive FMM total cost (in giga CPU cycles) as a function of the total number of particles, $N$, for the spiral case.}
\label{cost}
\end{figure}

The plot in Figure \ref{cost} is obtained after optimizing the subdividing threshold $t$. Basically, for each run, we have to tune $t$ to get the optimum cost. For a more comprehensive understanding of the optimization, in Figure \ref{thresh} we have illustrated the variation of the optimum threshold versus $N$ for two cases. The optimum threshold is not a unique number. In fact, there is an optimum interval rather than a single number. Any choice of $t$ within the optimum interval gives rise to the same adaptive tree with maximum height $l_{opt}$. This is the key parameter, which allows minimizing the computational cost. We will discuss this more in \S \ref{optimization}.

The extrema of the intervals in Figure \ref{thresh} lie along lines with different slopes. The ratio of slopes for consecutive lines is a constant number. This ratio depends on the distribution of particles. For instance, in the uniform case this ratio is $2^3$. However, for the spiral case it is $2^d$, where $d<3$. Each group of intervals in Figure \ref{thresh} share the same $l_{opt}$.

For each distribution, there are three key values related to the threshold: $t_{min}$, $t_{cut}$, and $t_{max}$, which are illustrated in the right plot in Figure \ref{thresh}. $t_{min}$ is the minimum acceptable threshold to get the optimum cost among different values of $N$. Similarly, $t_{max}$ is the maximum acceptable optimum threshold.  $t_{cut}$ belongs to the intersection of all of the optimum threshold intervals. $t_{cut}$ can also be interpreted as the maximum (over different values of $N$) of the beginning point of the optimum threshold intervals. The optimum threshold interval for each $l_{opt}$ starts from $t_{min}$, and ends at $t_{max}$. The optimum interval shifts as we increase $N$; however, as soon as $t_{cut}$ lies outside the optimum interval, $l_{opt}$ increases by 1, and another group of intervals forms. Similar to the slope of bounding lines, we have ${t_{max}}/{t_{cut}} = {t_{cut}}/{t_{min}}=2^d$.

This new quantity, $d$, determines the behavior of the optimal threshold interval as $N$ increases. For a uniform distribution of points $d$ is 3, and for a surface of points in space $d$ is 2. It can take various (and not necessarily integer) values for different point distributions. Therefore, we have to look after a generalized form of dimension that is not restricted to integer numbers to be able to categorize different non-uniform distributions. Parameter optimization is possible afterwards. This is discussed in the next sections.

\begin{figure}[!htbp] \centering
\includegraphics[width=.35\textwidth]{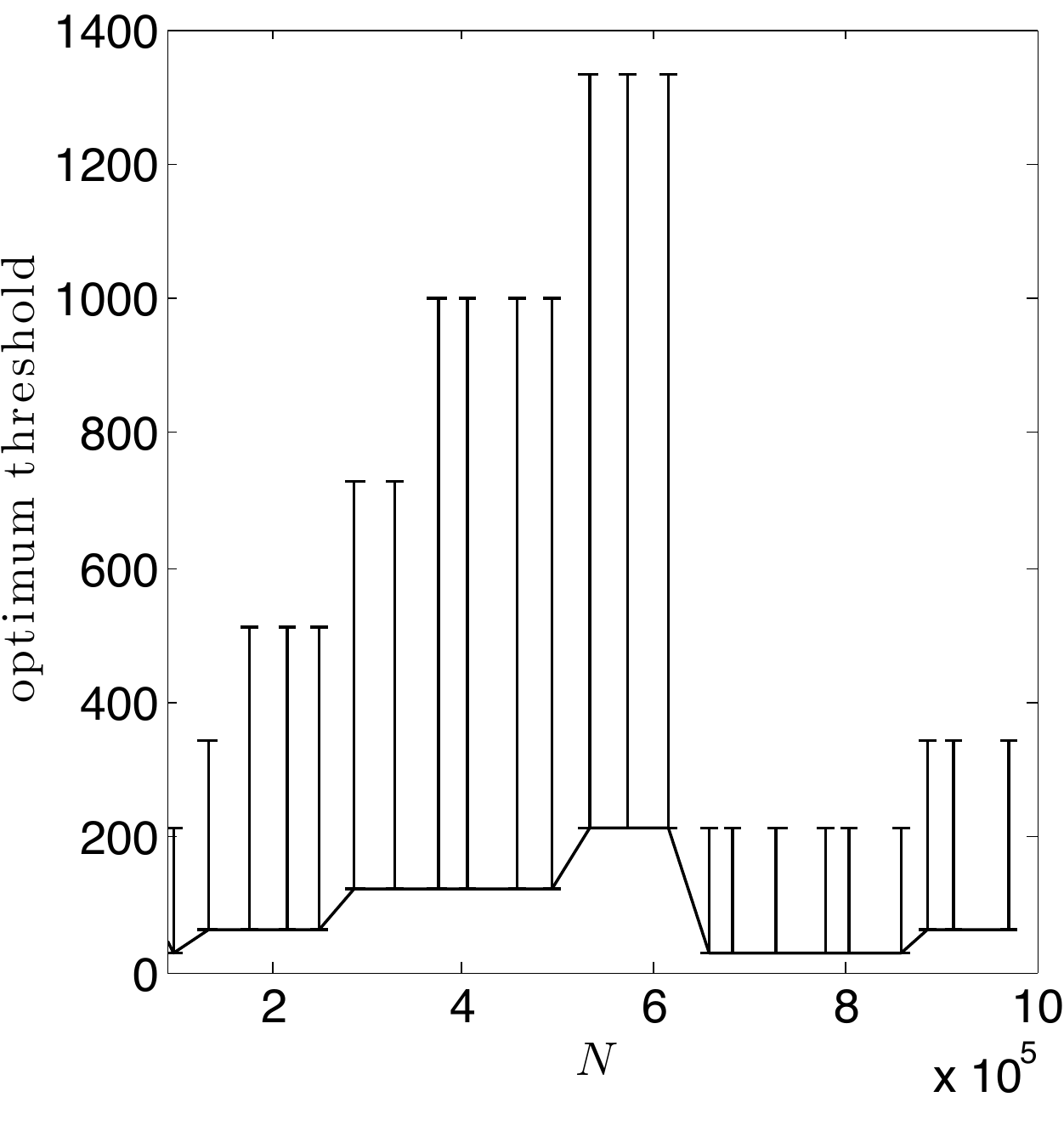}
\includegraphics[width=.37\textwidth]{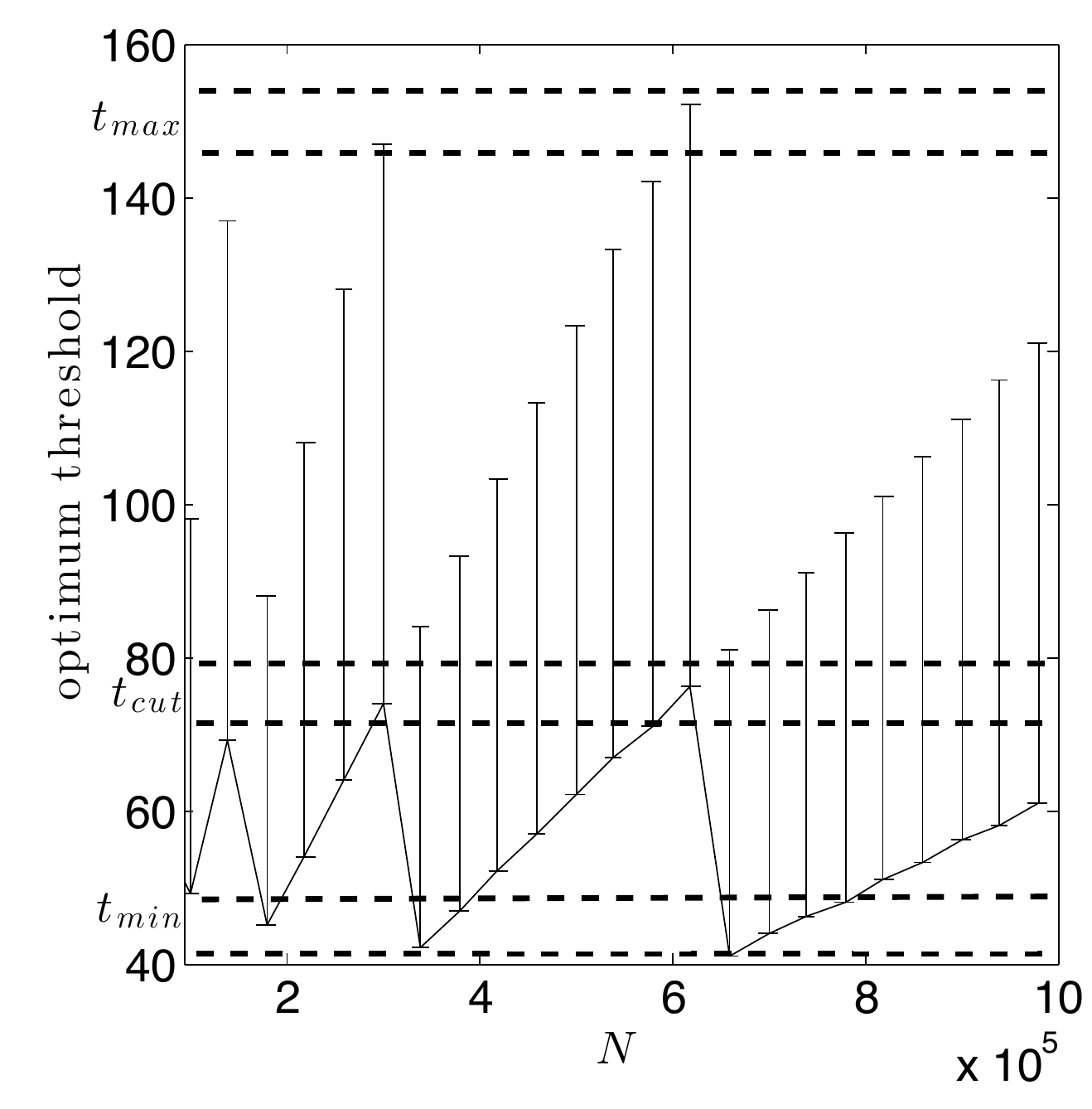}
\caption{Optimum subdividing threshold interval in the adaptive FMM as a function of $N$ for two particle distributions: Left: uniform; right: spiral.}
\label{thresh}
\end{figure}

\section {FMM and fractal sets} \label{fractal}

In \S \ref{numerical}, we illustrated how the study of different adaptive trees requires a generalized definition of dimension. Essentially, what governs the behavior observed in Figure \ref{thresh} is the number of \emph{occupied nodes} in level $l$ of the tree. For instance, for a full octree, there are always $2^{dl}$ occupied nodes in level $l$ of the tree, where $d=3$. In general, we are interested in the statistics of the number of occupied nodes in level $l$ of the octree. This is determined by the average number of children of a node. In order to formally define the average number of children of each node, we need to establish a strategy to increase the number of points (i.e., $N\to \infty$). As $N$ goes to $\infty$, its corresponding octree also grows. Therefore, we can look at the average number of children of nodes with depth $l$ as $l \to \infty$.  

There is a theoretical difficulty regarding taking the limit of a set of $N$ points as $N$ goes to infinity. If we simply consider a sequence of points $x_i$, and define the set $S_N = \{x_i\}_{i=1,\ldots,N}$, the sequence of sets $S_N$ can only converge to a set with a countable number of points. This won't work for our purpose. The Cantor set for example is uncountable. Therefore a different definition of the limit must be considered.

It turns out that it is easier and intuitively simpler to consider the limit of a sequence of octrees. To do this, we need to extend the definition of octree to trees with infinite levels as is defined below.

\begin{definition} (super octree)
A \emph{super octree} is a \emph{labeled} tree with infinite levels, where every node has exactly eight children. Each node is labeled as \emph{occupied} or \emph{empty}. A node is occupied only if its parent is occupied.
\end{definition}

A basic octree can be extended to a super octree (by infinitely subdividing) similar to the concept of extended tree introduced in \S \ref{proof}. Now, for a given super octree we can formally define the average number of children of nodes.

\begin{definition} (occupancy)
The \emph{occupancy} of a super octree $T$ is defined as:
\begin{equation}
\label{d_ocp}
d_{ocp}(T) \equiv \lim_{l \rightarrow \infty} \frac{\log_2(N_{ocp,l})}{l}
\end{equation}
where $N_{ocp,l}$ is the number of occupied nodes at level $l$.
\end{definition}

The above definition suggests that in level $l$ of the super octree, on average, $2^{d_{cop}l}$ out of $2^{dl}$ nodes are occupied. However, the above limit does not necessarily exist for a generic super octree. We are going to introduce a more general notion of dimension for super octrees that is always defined. To do this, we will associate a set of points in $\mathbb{R}^3$ to each super octree. The dimension of the super octree is then defined based on the set associated to it. In the following three definitions, the set of points associated with each super octree is defined step by step. 

\begin{definition} (path)
A \emph{path} ${\cal{P}}$ is an infinite sequence of nodes \{$N_i$\} starting from the root of the tree (i.e., $N_0$ is root), where $N_{i+1}$ is a child of $N_{i}$ for $i\ge0$. An \emph{occupied path} ${\cal{P}}$ is a path that only consists of occupied nodes.
\end{definition}


\begin{definition} (point associated to a path)
Let ${\cal{P}}$ be a path consisting of a sequence of nodes \{$N_i$\}, then this path is uniquely mapped to a point $\rho({\cal{P}}) \in [0, 1]^3 = \bigcap_{i} {\cal{I}}(N_i)$.
\end{definition}
Note that the intersection of nested closed sets is always non-empty. Since ${\cal{I}}(N_i)$'s are three-dimensional boxes with decreasing edge length $2^{-\delta(N_i)} = 2^{-i}$, their intersection is a single point.


\begin{definition} (associated set)
The \emph{associated set} of a super octree $T$ is defined as ${\cal{S}}(T)\subseteq [0, 1]^3= \bigcup_{{\cal{P}}_i} \rho({\cal{P}}_i)$, where ${\cal{P}}_i$ is an occupied path in $T$.

Conversely, for a given set of points $X \subseteq \mathbb{R}^3$ the \emph{associated super octree} $\mathcal{S}^{\dagger}$ can be created by performing the usual adaptive FMM subdividing process. $X$ should be first mapped to $[0, 1]^3$, and then the subdividing process is applied. However, some points may lie exactly on the midpoint of the interval in the subdividing process (e.g., points in $[0, 1]^3$ with a finite binary representation). To uniquely define the associated super octree, we can always pick the left (or always the right) interval when subdividing, in the case of midpoints. Hence, $\mathcal{S}$ is a \emph{pseudo inverse} for $\mathcal{S}^{\dagger}$ (i.e., $\mathcal{S}(\mathcal{S}^{\dagger}(X))=X$).
\end{definition}


Now, we can define the dimension of a generic super octree $T$ as the dimension of its associated set, ${\cal{S}}(T)$. But, what definition of dimension is the best choice? Let's begin with the \emph{Hausdorff dimension}, which is a generalized form of dimension. It is defined in two steps as follows \cite{dierk07}.

\begin{definition} (Hausdorff measure)
The $d$-dimensional \emph{Hausdorff measure}, $\mu_d$, of set of points $X$ for any $d$ in $R_0^+=[0, \infty)$ is defined as follows:
\begin{equation}
\label {H_measure}
\mu_d(X) = \lim_{\epsilon \rightarrow 0} \, \inf_{\mathcal{U}_{\epsilon}} \, \sum_{U_i \in \mathcal{U}_{\epsilon}} \left( \mbox{diam} \left( U_i \right) \right)^d,
\end{equation}
where the infimum is taken over all countable covers $\mathcal{U}_{\epsilon} = \{U_i\}_{i \in \mathbb{N}}$ of $X$ such that diam$(U_i)<\epsilon$ for all $i$. 
\end{definition}

\begin{lemma}
(Hausdorff dimension)
\label{Hausdorff}
For every bounded set $X$ in a given metric space, there is a unique value $d_H(X) \in R_0^+ \cup \{\infty \}$ such that $\mu_{d'} (X)=0$ if $d'>d_H(X)$ and $\mu_{d'} (X)=\infty$ if $d'<d_H(X)$. $d_H(X)$ is called the \emph{Hausdorff dimension} of $X$. 
\end{lemma}

For example, the Hausdorff dimension of any countable set is 0, and the Hausdorff dimension of $\mathbb{R}^n$, where $n \in \mathbb{N}$, is $n$.

We will need the following lemma to prove Corollary~\ref{H_OCP}.

\begin{lemma} (Hausdorff measure for super octrees)
\label{H_measure_super_octree}
Let $T$ be a super octree. Then:
\begin{equation}
\mu_d({\cal{S}}(T)) = \lim_{l \rightarrow \infty} \, \inf_{{\cal{N}}_l} \, \sum_{N_i \in \mathcal{N}_l} 2^{-d \, \delta(N_i)}
\end{equation}
where the infimum is taken over all countable set of occupied nodes $\mathcal{N}_l = \{ N_i \}_{i \in \mathbb{N}}$ of $T$, such that $\delta(N_i) \ge l$ for all $i$, and $\{ {\cal{I}}(N_i) \}_{N_i \in \mathcal{N}_l}$ is a cover for ${\cal{S}}(T)$. Note that length of the edges of the cube ${\cal{I}}(N_i)$ is  $2^{-\delta(N_i)}$, where $\delta(N_i)$ denotes the depth of the node $N_i$ in $T$.
\end{lemma}

The next corollary shows that the Hausdorff dimension is always a lower bound for the occupancy (when the latter is defined).

\begin{corollary}
\label{H_OCP}
For any super octree $T$:
\begin{equation}
d_H({\cal{S}}(T)) \le d_{ocp} (T)
\end{equation}
assuming the existence of $d_{ocp} (T)$.
\end{corollary}
\begin{proof}
If ${\cal{S}}(T)$ is a finite set of points, $d_H({\cal{S}}(T)) = d_{ocp} (T)=0$. So, let's assume ${\cal{S}}(T)$ includes an infinite number of points. From the definition of infimum, for any level $l$:
$$
N_{ocp,l} \, 2^{-d \, l} = \sum_{N_{i,l}} 2^{-d \, l} \ge  \inf_{{\cal{N}}_l} \, \sum_{N_i \in \mathcal{N}_l}  2^{-d \, \delta(N_i)} 
$$
where $N_{i,l}$'s are all the occupied nodes of level $l$ of the tree. Using Lemma~\ref{H_measure_super_octree} :
\begin{align*}
& \Rightarrow \quad \lim_{l \rightarrow \infty} N_{ocp,l} \, 2^{-d \, l} \ge \lim_{l \rightarrow \infty} \,  \inf_{{\cal{N}}_l} \, \sum_{N_i \in \mathcal{N}_l}  2^{-d \, \delta(N_i)} = \mu_d \left( {\cal{S}} \left( T \right) \right) \\
& \Rightarrow \quad \lim_{l \rightarrow \infty} N_{ocp,l}  \, 2^{-d \, \frac{\log_2(N_{ocp,l})}{d_{ocp}(T)}} \ge \mu_d \left( {\cal{S}} \left( T \right) \right) \\
& \Rightarrow \quad  \lim_{l \rightarrow \infty} N_{ocp,l}^{ \left( 1-\frac{d}{d_{ocp} \left( T \right) } \right)} \ge \mu_d \left( {\cal{S}} \left( T \right) \right)
\end{align*}
Since $N_{ocp,l} \rightarrow \infty$ as $l \rightarrow \infty$, we can conclude:
$$
\mu_d \left( {\cal{S}} \left( T \right) \right) =0 \hspace{4mm} \mbox{for } d>d_{ocp}(T)
$$
Hence, based on Lemma \ref{Hausdorff}, $d_H \left( {\cal{S}} \left( T \right) \right) \le d_{ocp} (T)$. 
\qed
\end{proof}

Corollary \ref{H_OCP} provides a lower bound for our parameter of interest $d_{ocp}(T)$. What we described as occupancy is in fact equivalent to another notion of dimension called \emph{box counting dimension} or \emph{capacity}. The definition of capacity is similar to the Hausdorff dimension except that it only allows coverings with boxes of the same size. For \emph{uniformly self-similar} sets, the Hausdorff dimension and capacity are numerically equal (see~\cite{hutchinson81}). The numerical equality is also shown for a broader family of sets, namely multi-scale fractals (see~\cite{ronnie93}). However, for a general set, the Hausdorff dimension and capacity are not equal. In fact, there are many sets where capacity is not defined while the Hausdorff dimension is always defined \cite{dierk07}. The next three examples show different possible scenarios. The results below are given without proof.

\begin{example}
Let $X$ be the set of all rational points in $[0,1]$. Since $X$ is a countable set, the Hausdorff dimension is 0. However, since $X$ is dense, the super binary tree corresponding to $X$ is full (i.e., all nodes are occupied); therefore, its occupancy (or box counting dimension) is 1.
\end{example}

\begin{example}
\label{gen_Cantor_ex}
Let $X$ be the generalized Cantor set defined as follows. Start with segment $[0, 1]$, remove an \emph{open} segment with length $\gamma \in (0,1)$ from the middle. In the next step, remove the middle segments of the two new subintervals with length ratio $\gamma$, and continue. This is illustrated in Figure \ref{gen_Cantor}. The famous middle-third Cantor set corresponds to the case $\gamma=1/3$. Assume $X$ corresponds to a super octree $T$ (i.e., ${\cal{S}}(T)=X$). For the generalized Cantor set, the Hausdorff dimension and occupancy are equal:
\begin{equation}
\label{d_gen_cantor}
d_H(X)=d_{ocp}(T) = -\frac{\log(2)}{\log(\frac{1-\gamma}{2})}
\end{equation}
\end{example}

\begin{figure}[!htbp] \centering
\includegraphics[width=.35\textwidth]{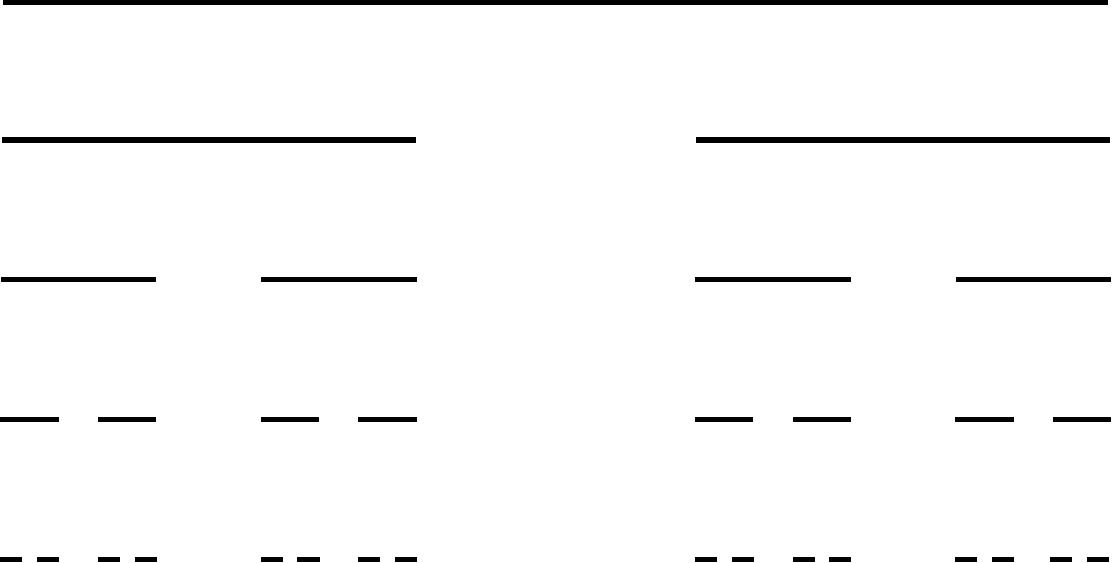}
\caption{Generalized Cantor set}
\label{gen_Cantor}
\end{figure}

\begin{example}
Let's reconstruct $X$ in the Example \ref{gen_Cantor_ex} with the sequence of scales $0 < \gamma_i < 1$. At step $i$, replace all intervals with two subintervals by removing a segment of length ratio $\gamma_i$. If $\gamma_0 = \gamma_1= \gamma_2= \ldots = \gamma$, the generalized Cantor set results. However, if the sequence $\{\gamma_i\}$ does not converge, the box-counting dimension is not defined, while the Hausdorff dimension is always well-defined \cite{dierk07}.
\end{example}

Now, we enter a key part in our formal definitions. The above framework enables us to analyze the dimension of super octrees associated to infinite number of points. In practice, we need to work with a finite set consisting of $N$ points, and study the point distribution effect as $N \to \infty$. If we consider a sequence of finite sets consisting $N$ points, $S_N$, as $N \to \infty$ we can only converge to a set with a countable number of points. Instead, we will consider a converging sequence of octrees. For example, consider the interval $(0,1)$. This is an uncountable set. The super octree for this set is full. We can consider the sequence of octrees $T_i$ obtained by simply keeping all the nodes in the super octree down to level $i$. As $i$ goes to infinity, we are correctly modeling $N \to \infty$ for a uniform distribution of points in the FMM. Taking the limit of a sequence of octrees is therefore the correct concept for our purpose.

The overall process is then as follows. We consider an infinite set of points, such as the Cantor set. Then, we consider the associated super octree $T$. We build a sequence of octrees $T_i$ that converges to $T$. For each $T_i$, we can consider an appropriate distribution of points such that the octree for that set is exactly $T_i$. As $i$ goes to infinity, the number of points in the FMM also goes to infinity. The convergence of the sequence of octrees is defined as follows.

\begin{definition} (converging sequence of octrees)
A sequence of octrees $\{ T_i \}$ converges to a super octree $T_s$, if for any $l>0$ there is $n_l>0$ such that if $i>n_l$, then $T_i$'s are identical from the root up to level $l$.
\end {definition}
\begin{remark}
If $c_l=\frac{\log_2 N_{ocp,l}}{l}$, where $N_{ocp,l}$ is the number of occupied nodes at level $l$ of octree $T_i$ with $i>n_l$, then $\lim_{l \rightarrow \infty} c_l = d_{ocp} (T_s)$ upon existence of $d_{ocp} (T_s)$.
\end{remark}

Note that the details of the point distribution is not important as long as the sequence of octrees is correct (i.e, converges to the desired super octree). This is the main data structure that determines the complexity of the FMM. For example, the location of a point in a leaf is irrelevant as far as the computational cost of the method goes. From there, depending on the dimension, we will be able to determine the optimal parameters for the FMM as well as discuss the running time of the algorithm.

For the purpose of the adaptive FMM study, we are looking for a set of points, $X$, leading to interesting FMM trees. The generalized Cantor set provides such properties for us. Generalized Cantor sets in $\mathbb{R}^n$ are \emph{uncountable}, \emph{perfect}, and \emph{nowhere dense}. Recall that a perfect set is a closed set with no isolated points, and a set that is nowhere dense is a set whose closure has an empty interior. The interpretation of these properties in the context of adaptive FMM are explained in Lemmas \ref{perfect} and \ref{nowheredense}. Moreover, the family of generalized Cantor set provides the full range of Hausdorff dimensions, from 0 to 3.

Let ${\cal{C}}_{\gamma}$ to be a one-dimensional generalized Cantor set with parameter $\gamma \in (0,1)$. Each point in this set can be described as follows:
$$
x \in {\cal{C_{\gamma}}} \iff x=\sum_{i=0}^{\infty} d_i b^i \hspace{4mm} \mbox{where } d_i \in \{0,a\} \mbox{ with } a=\frac{1+\gamma}{2} \mbox{ and } b=\frac{1-\gamma}{2}
$$
We can map ${\cal{C}}_{\gamma}$ to $[0, 1]$ using following surjection:
$$
f(x)=f(\sum_{i=0}^{\infty} d_i b^i)=\sum_{i=0}^{\infty} \frac{d_i}{a} 2^{-i}
$$
This mapping readily shows why Cantor set is uncountable, perfect, and nowhere dense.

\begin{remark}
An $n$-dimensional (topologically) generalized Cantor set can be obtained by $n$ times \emph{direct product} of a one-dimensional ${\cal{C}}_{\gamma}$ with itself. Therefore, for the three-dimensional case:
$$
{\cal{C}}_{\gamma} \times {\cal{C}}_{\gamma} \times {\cal{C}}_{\gamma} \subseteq [0, 1]^3 \hspace{4mm} \mbox{and} \hspace{4mm} d_H({\cal{C}}_{\gamma} \times {\cal{C}}_{\gamma} \times {\cal{C}}_{\gamma}) =-3\frac{\log(2)}{\log(\frac{1-\gamma}{2})}
$$
\end{remark}


\begin{lemma} (octree of a perfect set)
\label{perfect}
Let $X \subseteq [0, 1]^3$ be a perfect set. In the octree for the adaptive FMM on $X$, there is no occupied node with a finite number of particles in it.
\end{lemma}


\begin{lemma} (octree of a nowhere dense set)
\label{nowheredense}
Let $X \subseteq [0, 1]^3$ be a nowhere dense set. In the octree for the adaptive FMM on $X$, there is no full subtree (i.e., a subtree whose nodes are all occupied).
\end{lemma}

Before moving on to the next section, and study the role of $d_H$ in the performance of the adaptive FMM, let's present one dramatic example to demonstrate why the \emph{modified adaptive FMM algorithm} introduced in \S \ref{proof} is necessary. 
\begin{example}
\label{crazy_dist}
Define $X \subseteq [0, 1]$ as follows. Start with $[0, 1]$. At step $i$ consider the current set of intervals. Take each interval and subdivide it into two subintervals. Take each interval and reduce its length by $2^{2^i}$, that is, if the interval starts from $a$ and has length $l$, the scaled interval starts from $a$, and has length $\frac{l}{2^{2^i}}$. The intersection of all of these intervals is $X$. Basically:
\begin{gather*}
x \in X \iff x\mbox{'s binary form is } x=0.d_1 d_2 d_3 d_4 \dots \\
\mbox{ where } d_1\in \{0,1\}, d_2=d_3=0, d_4\in \{0,1\}, d_5=d_6=d_7=d_8=0,  d_9\in \{0,1\}, \ldots
\end{gather*}
We can show that $X$ is uncountable, perfect, and nowhere dense.

Now, the arbitrarily long sequences of $0$'s guarantees the existence of arbitrarily long branches of singleton nodes in the associated adaptive binary tree. Hence, the modified adaptive FMM algorithm is necessary if one wants to get linear complexity for such a set. Indeed, if we consider the intervals created at step $i$, we have $N=2^i$ such intervals (assume that we create a set with $2^i$ points by picking a random point in each of the intervals obtained at step $i$). But, we have $\Omega(N)$ singleton branches, each of size $\Omega(N)$. Therefore, without trimming the singleton branches, the overall cost of the FMM is $\Omega(N^2)$ in that example.
\end{example}


\section {Optimal choice of the FMM parameters} \label{optimization}

In this section, we are going to use the generalized Cantor set introduced in \S \ref{fractal}, to study the computational cost of the FMM. The same implementation as described in \S \ref{numerical} is used. 

In the previous section, we explained how to construct a sequence of octrees that is converging to the desired super octree. Here, we consider the super octree associated with the generalized Cantor set ${\cal{C}_{\gamma}}$. To construct a finite set of points, we truncate the number of levels to a finite number, and pick one point in each leaf node of the resulting tree. For example, in Figure \ref{gen_Cantor}, we can pick 16 points from intervals resulted at step 4. Figure \ref{Cantor_4096} shows the clusters of the octree corresponding to a subset of ${\mathcal{C}_2}$ with 4096 points.

Earlier in the paper, we discussed the subdividing threshold $t$. This number can take any value between $1$ and $N$. The former results in a calculation with highly clustered particles, whereas the latter brings about direct $\mathcal{O}(N^2)$ computation. Its optimum value depends on many  parameters such as machine architecture, point distribution, low rank approximation order, etc. In \cite{zorin03}, the optimum threshold is picked based on the low rank approximation order. Chandramowlishwaran et al.~\cite{aparna10} showed how the total cost varies with the choice of threshold, and picked the optimum value by tuning. Gumerov et al.~\cite{gumerov03} discussed the analytical optimal value of the subdividing threshold for the uniform distribution of particles. 

We will assume a constant low-rank approximation order (the number of Chebyshev points in our case), and study the variation of the total cost, and the optimum threshold as a function of the Hausdorff dimension for the generalized Cantor sets. 

\begin{figure}[!htbp] 
\centering
\includegraphics[width=.8\textwidth]{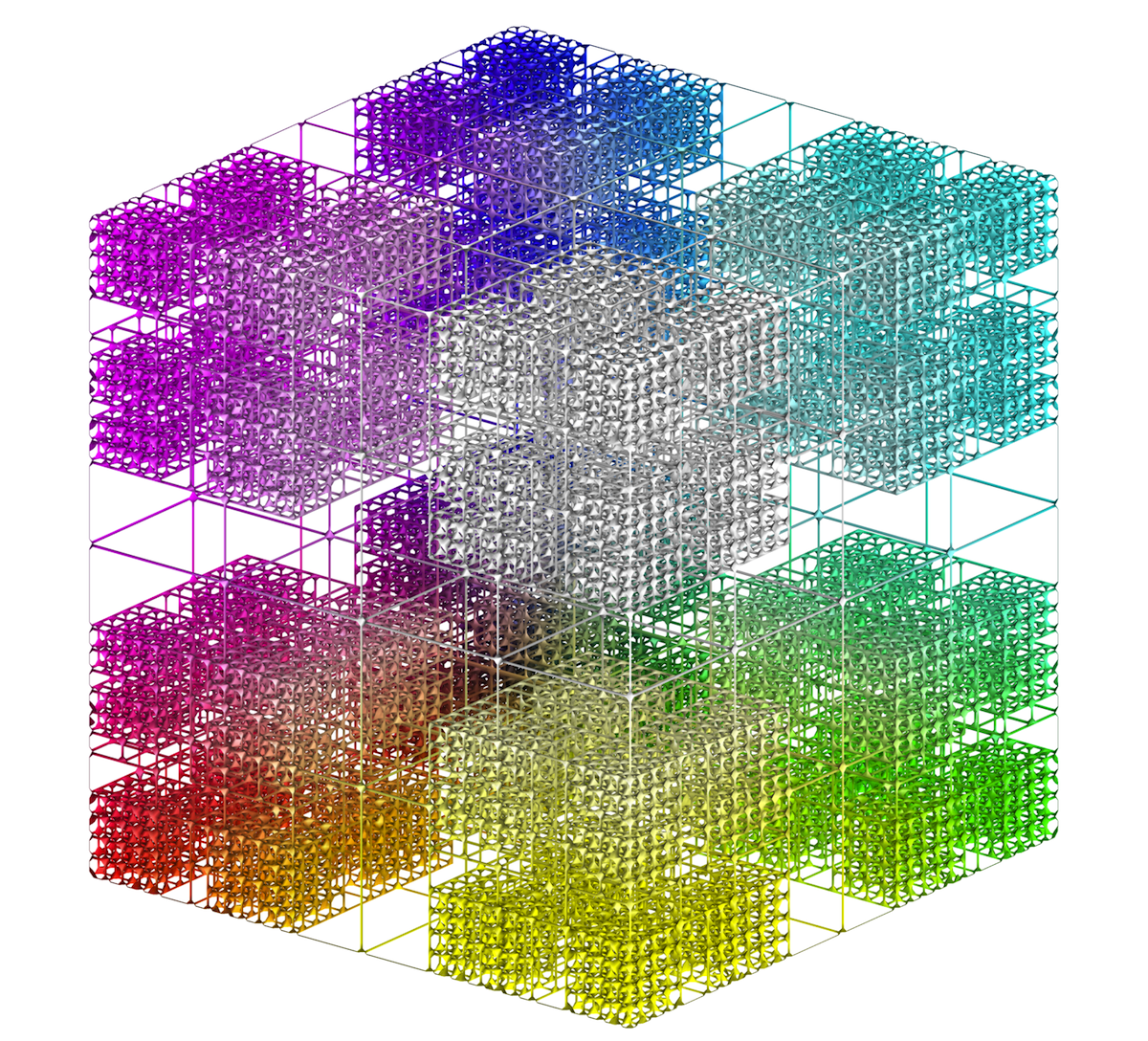}
\caption{Three dimensional Cantor set FMM tree, for $N=4096$, $t=1$, and $d_H=2$. For visualization purpose it is colored by [R,G,B]=$[x,y,z]$.}
\label{Cantor_4096}
\end{figure}

\subsection{New double-threshold method} \label{new_method}

The subdivision threshold $t$ is not the only important parameter for optimizing the computational cost in FMM. There are in fact two separate criteria that can be used. One is $t$, and the other is the maximum number of levels (i.e., the maximum depth of the tree). In the literature, the maximum number of levels is obtained simply by fixing $t$, and then subdividing all nodes until all leaf nodes have no more than $t$ points. For a non-adaptive FMM, where all of the leaf nodes have the same depth, it is possible to find one optimal subdividing threshold $t$ \cite{gumerov03}. However, for the adaptive case this cannot be optimal. For example, in Figure \ref{1dCantor} a one-dimensional general Cantor set is partially illustrated. Unlike the uniform case, the leaf nodes have different depths, and are sparsely distributed. 

Let's consider all the operators involved in the FMM. Subdividing a node results in the following. If the node is at level $\lmax$, we increase M2L, M2M and L2L while reducing P2P. If the node is at a higher level in the tree, we also increase M2L, M2M and L2L, but then reduce the cost of the P2L and M2P operators. The cost of the P2L/M2P operator is not the same as P2P. Therefore, we should expect that the optimal threshold at $\lmax$ is different from the optimal threshold higher up in the tree.

\begin{figure}[!htbp] 
\centering
\includegraphics[width=.8\textwidth]{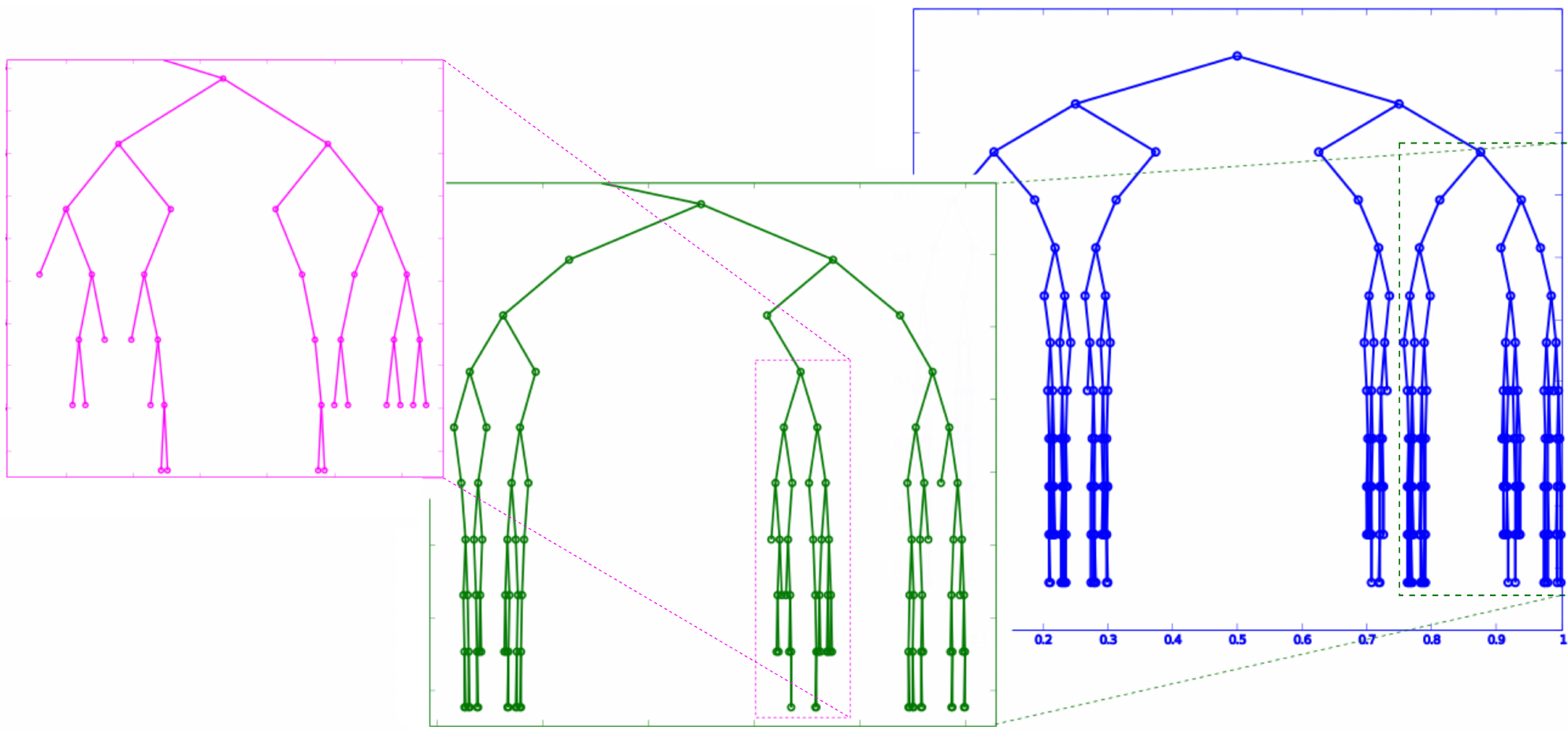}
\caption{One-dimensional general Cantor set with $\gamma = 0.4$. The figure on the right shows the entire tree. The middle figure in green is a zoomed-in copy of the right side of the tree. The left figure in magenta is zoomed-in again, showing the ``second main branch'' in the green tree. This figure shows the adaptive nature of the FMM tree for a fractal set.}
\label{1dCantor}
\end{figure}

The key point is that the M2L operator, depending on the implementation, is relatively cheap compared to P2P. The M2L operator is precomputed, whereas in P2P calculation we need kernel evaluation. M2L can also be accelerated (not done in this paper though) using a singular value decomposition (in a general case), or a fast Fourier transform using uniformly distributed points instead of Chebyshev points.\footnote{There are stability and accuracy issues with this approach. These can be successfully solved but we won't discuss this point in this paper.} As a result, the optimal threshold at $\lmax$ must be small, meaning that it is smaller than the number of Chebyshev nodes used for the multipoles and locals.

In an M2P operator, we need to compute the interaction between particles and Chebyshev points. Therefore, we expect M2P (resp.\ P2L) to be more expensive than P2P for the same threshold. The threshold at $l<\lmax$ should consequently be lower than the threshold at $\lmax$.

With this insight, we decided to investigate the following subdivision rule. We choose $\lmax$ and $t$, and then we:
$$
\text{subdivide a node iff: }~~~
\text{level of node} < \lmax \quad \text{and} \quad \text{number of points in node} > t
$$
This condition implies that it is possible for leaf nodes to have more than $t$ points. However, all clusters that are not at the level $\lmax$ must have less than $t$ points. Therefore, we have two parameters to tune in this approach. This is the new method we investigated. Given $\lmax$, we define $s(\lmax)$ to be the maximum number of points per leaf node. In Figure~\ref{sl}, for the class of generalized Cantor sets, $s(\lmax)$ is plotted. In this case, $s(\lmax)$ can be approximated by $N/2^{ld_H}$, where $d_H$ is the Hausdorff dimension of the general Cantor set.

\begin{figure}[htbp] \centering
\includegraphics[width=.5\textwidth]{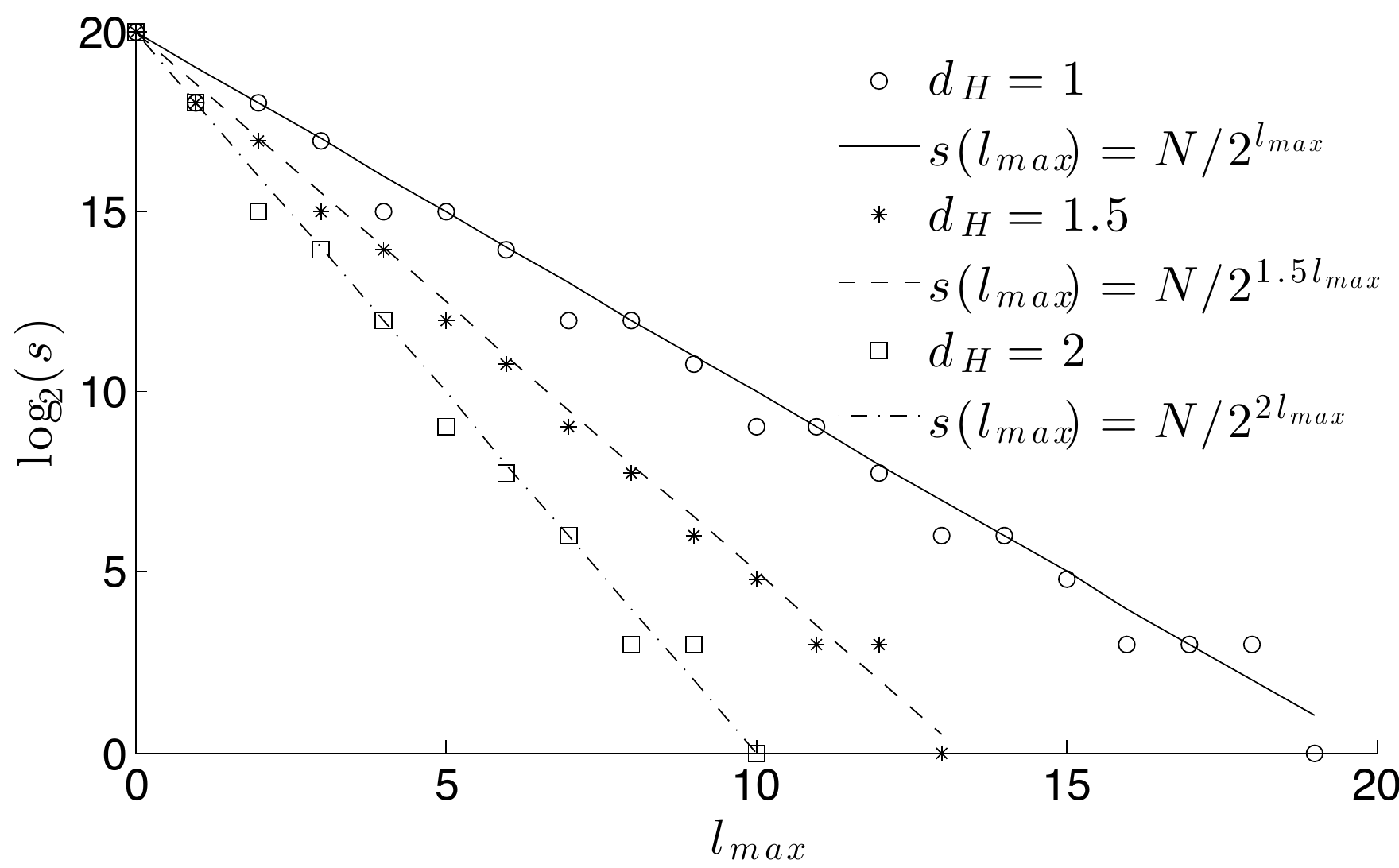}
\caption{$s(\lmax)$ (the maximum number of particles per leaf in a tree with at most $\lmax$ levels) as a function of $\lmax$ for different $d_H$ and $N=10^6$.}
\label{sl}
\end{figure}

In our approach we have two parameters to tune: $t$ and $\lmax$. The goal is to pick parameters such that the total cost of the calculation, $\cost(\lmax,t)$, gets its optimal value. In the conventional subdividing scheme, the optimization is restricted only to the parameter pairs $(\lmax,s(\lmax))$. 

For the family of general Cantor sets, we investigated the total FMM cost for different pairs of parameters $(\lmax,t)$. Decreasing the value of $t$, while $\lmax$ is constant, is similar to the concept of extended tree introduced in \S \ref{proof}. Essentially, by decreasing $t$ we are substituting M2P and P2L operations, which are fairly expensive, by new M2L operations that are cheaper. It turns out that the total cost changes slightly at small values of $t$ for different $\lmax$. Therefore, settling on $t=1$, and search for the optimal $\lmax$ is a simple near-optimal choice. 

In Figure~\ref{cost_t}, the total cost as a function of $t$ for $d_H=1$ and $d_H=2$ is shown. To obtain these plots, at first, with $t=1$ we found the optimal maximum depth $\lmax$. Then for the optimal $\lmax$, we changed the subdividing threshold from $t=1$ to $t=s(\lmax)$. We see that the choice of threshold $t=1$ is very close to the optimum. This is the case for different values of $d_H$. In this figure, we observe the rapid increase in M2P and P2L as $t$ increases. This agrees with our analysis that small values of $t$ are preferable.

\begin{figure}[htbp] \centering
\includegraphics[width=.4\textwidth]{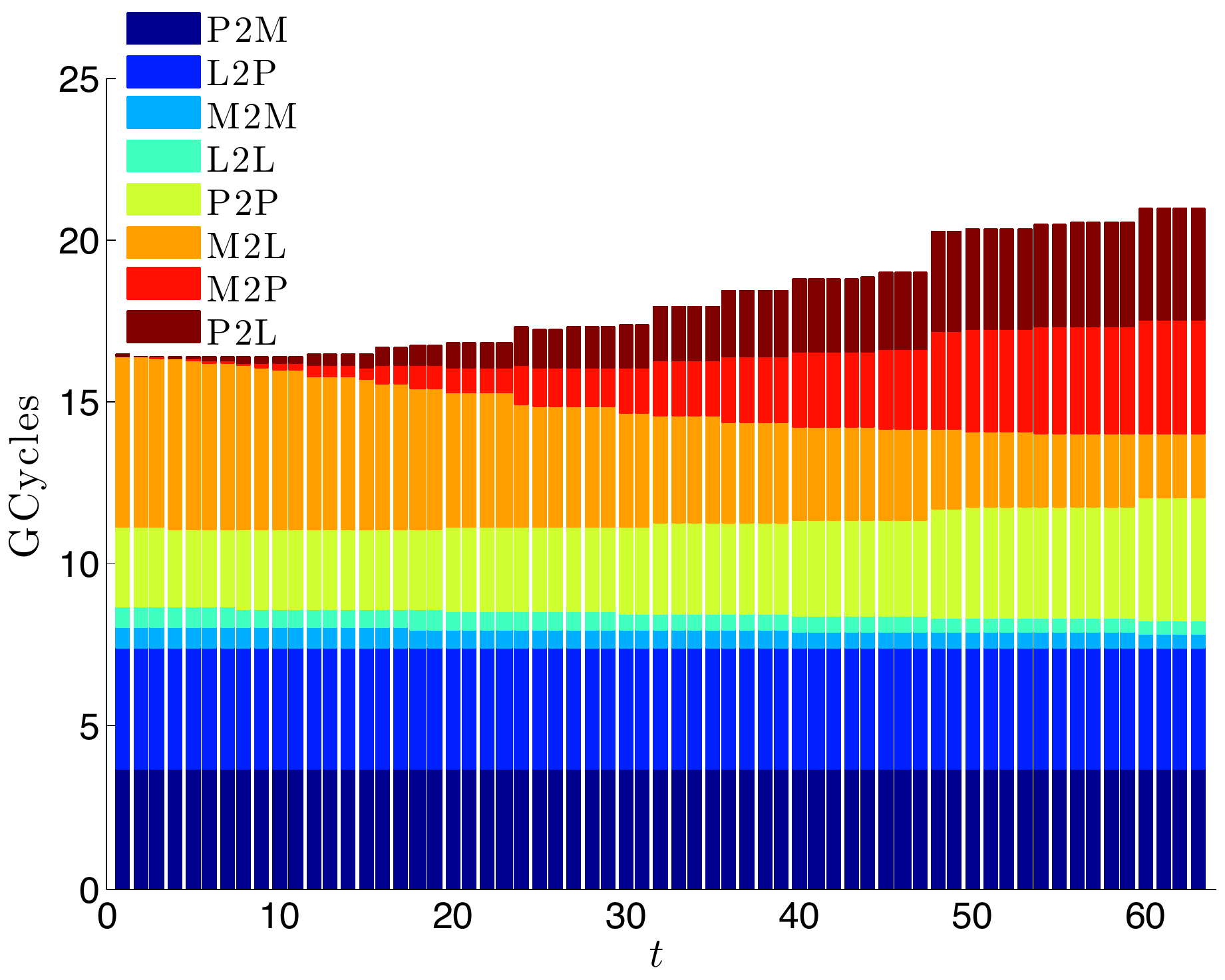}
\includegraphics[width=.4\textwidth]{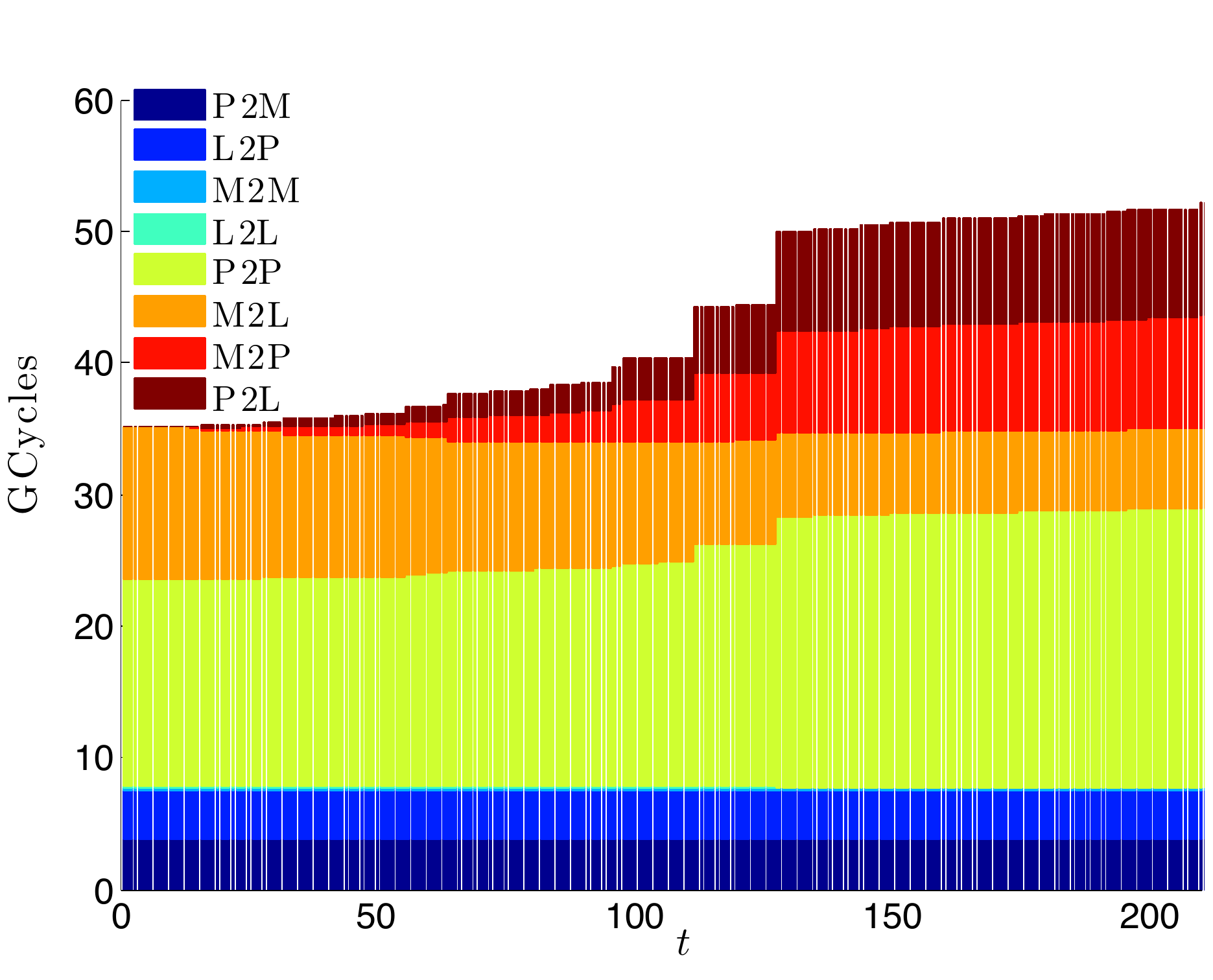}
\caption{The adaptive FMM total cost for $N=10^6$ particles with general Cantor set distributions. We picked $l_{opt}$ using the new subdividing method. Each sub-algorithm cost is shown as $t$ increases from 1 to $s(\lmax)$. Left: $d_H=1$. Right: $d_H=2$. This figure shows that small values of $t$ are preferable. This is caused by the rapid increase in the cost of M2P and P2L.}
\label{cost_t}
\end{figure}

In Figure~\ref{comparison}, we have plotted the optimal total cost as a function of $d_H$ for the conventional and new subdividing schemes. Note that the total cost function takes its optimal value at different $\lmax$'s for each scheme. Observe that the new and the conventional threshold schemes become the same as $d_H \rightarrow 3$, since in that case all leaves are lying in the finest level. 

\begin{figure}[htbp] 
\centering
\includegraphics[width=.4\textwidth]{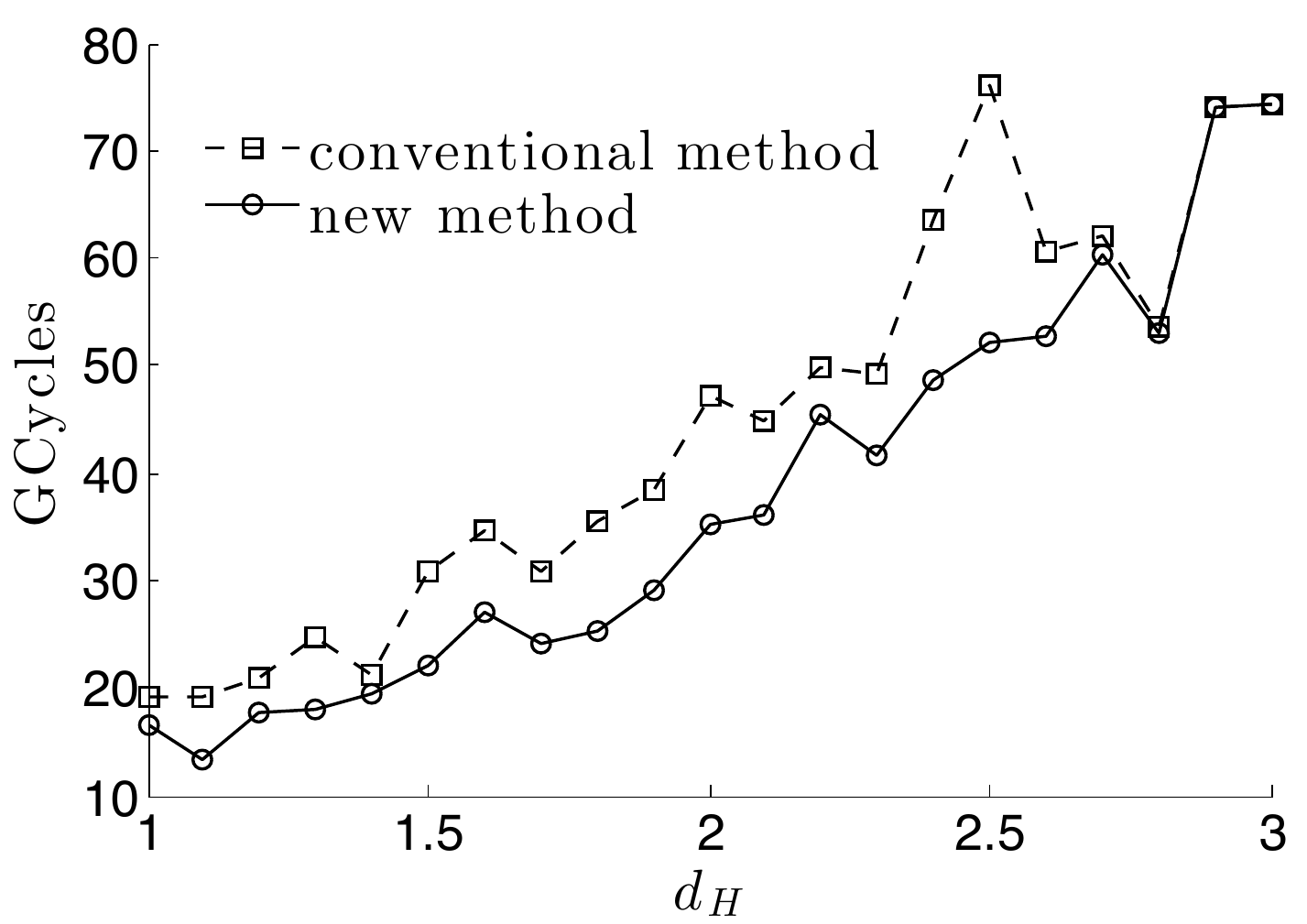}
\caption{Comparison between the optimal cost of the conventional subdividing method and the new method. We note that a greater improvement would be obtained with our new method if we had implemented an optimized M2L operator. For $d_H=3$, both methods become identical since all leaf nodes are at level $\lmax$.}
\label{comparison}
\end{figure}

From Figure \ref{cost_t}, we see that for $t=1$, the major part of the total cost is due to M2L, P2P, P2M, and L2P operators. Since $N$ and $r$ are fixed, the P2M and L2P costs are fixed. In the next part, we have introduced an {\it ad hoc} analysis to find the relation between the computational cost and $d_H$, keeping everything else fixed (i.e., $N$, $r$, the computer hardware, etc).

\subsection {Heuristic optimization analysis} \label{heuristic}

In \S \ref{new_method}, we showed that the M2L and P2P operators are the key operators that determine the total cost, while other operators have either constant or small cost. In this section by total cost we refer to the sum of the cost of P2P and M2L operations. 

Let's begin with the following assumptions using the definition of \emph{occupancy} (which happens to be the same as Hausdorff dimension for the family of fractal sets that we are studying). Assume that the maximum depth of the tree is $l$.
\begin{itemize}
\item Number of occupied nodes with depth $i\simeq 2^{d_H i}$
\item Number of leaves $\simeq 2^{d_H l}$
\item Number of particles per leaf $\simeq N/2^{d_H l}$
\item Number of P2P interactions  per leaf $\simeq 3^{d_H}$
\item Number of M2L interactions per node $\simeq 6^{d_H} - 3^{d_H}$
\item Total number of occupied nodes in the tree $\simeq (2^{d_H})^0+(2^{d_H})^1+(2^{d_H})^2+\ldots+(2^{d_H})^l = \frac{(2^{d_H})^{l+1} -1}{2^{d_H}-1}$
\end{itemize}
Based on the above assumptions, the P2P+M2L cost is:
\begin{equation}
\label{tot_COST}
\cost\simeq \alpha 2^{d_H l}  (\frac{N}{2^{d_H l}})^2  3^{d_H} + \beta \frac{2^{d_H l +d_H}-1}{2^{d_H}-1}  (6^{d_H} - 3^{d_H})
\end{equation}
where, $\alpha$ and $\beta$ depend on other parameters, such as machine architecture, low rank approximation, etc. We are going to optimize the above function with respect to $l$. Set $\frac{\partial}{\partial l} \cost =0$:
\begin{equation}
\label{l_opt}
l_{opt} \simeq \frac{\log_2 (\frac{\alpha}{\beta})}{2d_H}+\frac{\log_2 N}{d_H} - \frac{1}{2} 
\end{equation}
Plug $l_{opt}$ in Equation (\ref{tot_COST}) to obtain $\cost_{opt}$:
\begin{equation}
\label{opt_cost}
\cost_{opt} \simeq 3^{d_H} \left( 2\sqrt{\alpha \beta} N 2^{d_H/2} -\beta \right)
\end{equation}
Note that for $l=l_{opt}$ we have $\cost_{\mbox{P2P}} \simeq \cost_{\mbox{M2L}}$. For large values of $N$, we can ignore the second term in the Equation (\ref{opt_cost}), and rewrite it as: $\cost_{opt} \sim 2^{d_H/2} 3^{d_H} N$, which is consistent with our linear complexity expectation. From this heuristic analysis, we can conclude the following models.

{\bf Model of the optimum tree height:} The optimum height of the tree in the adaptive FMM with fractal sets is of the form $$l_{opt} \simeq \frac{K_1+\log_2 (N)}{d_H} +K_2,$$ where $K_1$ and $K_2$ are functions of other parameters.

{\bf Model of the optimum cost:} The optimum P2P+M2L cost in the adaptive FMM with fractal sets is of the form $$\log (\cost_{opt}) \simeq K_3+\log (N) +  1.445 \; d_H,$$ where $K_3$ depends on the other parameters.

In Figure~\ref{verification}, we have verified the above models, for generalized Cantor sets with Hausdorff dimension varying from $d_H=1$ to $d_H=3$. The predicted behavior was obtained with good accuracy. The least square line in the right plot of Figure \ref{verification} has slope of 1.2 which is close to the analytical prediction $\log(3) + \log(2)/2=1.44$ (this gets better for larger values of $N$). Also this plot shows that the worst case in the adaptive FMM is when the tree is full. It is consistent with the analysis presented in \S \ref{proof}, where we introduced linear complexity for adaptive FMM.

In \S \ref{fractal} we showed that the average number of children of a node is what determines the optimal parameters and optimal cost. This can be formalized by the Hausdorff dimension. Note that even with a finite number of points, the concept of average number of child nodes is well-defined, and can be used to tune the FMM using the above models.

\begin{figure}[!htbp] 
\centering
\includegraphics[width=.4\textwidth]{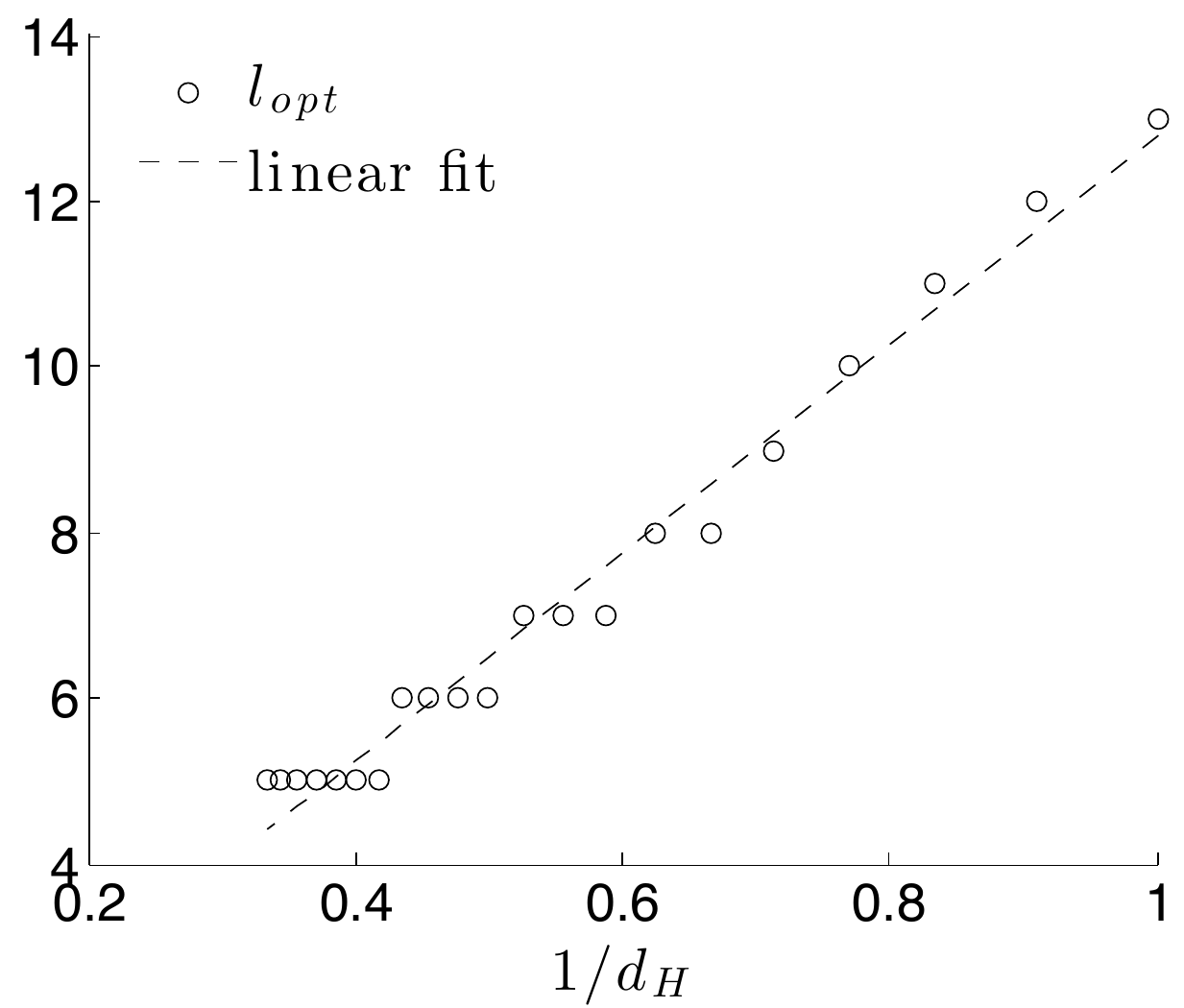}
\quad
\includegraphics[width=.4\textwidth]{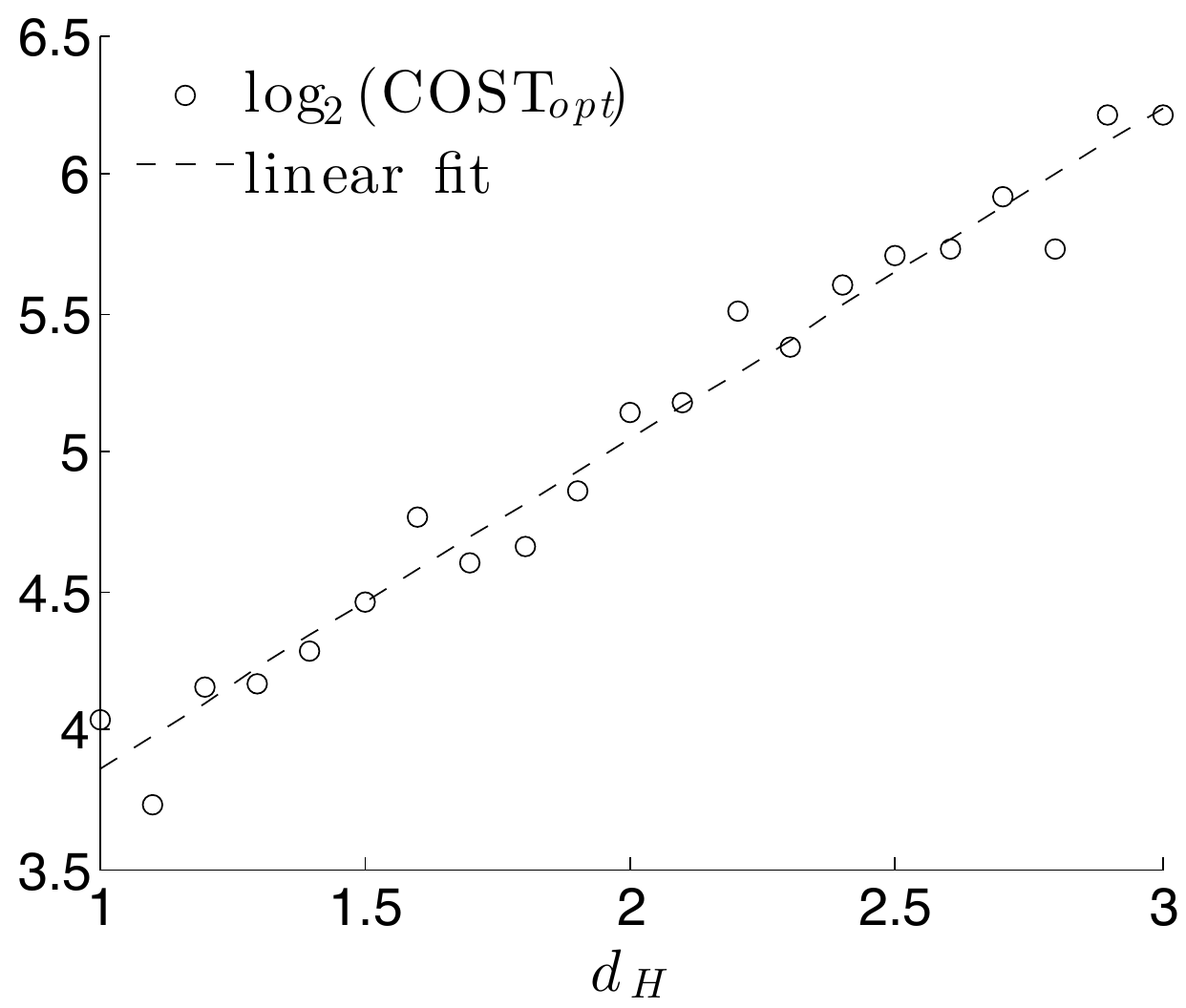}
\caption{Numerical verification for the heuristic analysis with $d_H$ varying from 1 to 3 for $N=10^6$. Left: variation of $l_{opt}$ as a function of $1/d_H$. Right: The optimum cost for the new proposed subdividing scheme. Almost linear variation of $\log (\cost_{opt})$ as a function of $d_H$ is shown.}
\label{verification}
\end{figure}


\section {Conclusion} \label{conclusion} 

In this paper, we provided a new proof of the linear complexity of the FMM in the general case. Note that our proof focuses on the cost of the FMM operators and does not consider the cost of building the FMM tree. We introduced the required modifications to the regular adaptive FMM in order to achieve ${\cal{O}}(N)$ complexity for arbitrary point distributions. The standard adaptive FMM may require arbitrarily long calculations for certain point distributions.

In order to study different adaptive FMM cases, we defined a process to properly add points and having $N \to \infty$. For this, we need to introduce the concept of super octree $T$ (an octree with an infinite number of levels) and establish a connection between a set of points and its associated super octree. This represents a significant extension of previous studies of the adaptive FMM, which are limited to a narrower range of types of distributions. We then focused on the fractal dimension of sets, e.g., the Hausdorff dimension, as a useful parameter to understand the complexity of the adaptive FMM and determine optimal parameters for the FMM.

We showed that the usual criterion to determine whether a node should be subdivided or not, based on a threshold for the number of points in leaf nodes, is insufficient. Indeed, using a single threshold $t$ throughout the tree is sub-optimal. At $\lmax$, the optimal $t$ is essentially found by finding a balance between P2P and M2L. Instead, at lower levels in the tree, the operators M2P and P2L need to be taken into account and those are typically significantly more expensive than the M2L operator. As a result, the threshold for leaves at $\lmax$ should be chosen differently from the threshold at lower depths. Consequently, we considered an optimization procedure with two parameters: the maximum number of levels $\lmax$ and a threshold $t$. The conventional optimization with a single threshold $t$ is strictly a special case of our approach.

We introduced a near-optimal solution to this new optimization problem. Using generalized Cantor sets, we showed that the new scheme has better performance compared to the conventional threshold method. The improvement would be even more evident if we had implemented an improved version of the M2L operator (for example accelerated using FFTs). Moreover, our near-optimal solution shows that using X-list and W-list is not necessarily beneficial in the implementation of the adaptive FMM. 

Finally, we established how the tree occupancy, which is the same as the Hausdorff dimension in the case of self-similar sets, affects the optimum parameters in the adaptive FMM. Based on this, we provided an optimization model to find the number of levels in the tree and the cost of the total calculation, as a function of the dimension of the point distribution. 

\bibliographystyle{plainnat}
\bibliography{manuscript}

\begin{thebibliography}{45}
\providecommand{\natexlab}[1]{#1}
\providecommand{\url}[1]{\texttt{#1}}
\expandafter\ifx\csname urlstyle\endcsname\relax
  \providecommand{\doi}[1]{doi: #1}\else
  \providecommand{\doi}{doi: \begingroup \urlstyle{rm}\Url}\fi

\bibitem[Abeyratne et~al.()Abeyratne, Manikonda, and Erdelyi]{abeyratne13}
S.~Abeyratne, S.~Manikonda, and B.~Erdelyi.
\newblock A novel differential algebraic adaptive fast multipole method.

\bibitem[Allain and Cloitre(1991)]{allain1991characterizing}
C.~Allain and M.~Cloitre.
\newblock Characterizing the lacunarity of random and deterministic fractal
  sets.
\newblock \emph{Physical review A}, 44\penalty0 (6):\penalty0 3552, 1991.

\bibitem[Aluru(1996)]{aluru1996greengard}
Srinivas Aluru.
\newblock Greengard's n-body algorithm is not order n.
\newblock \emph{SIAM Journal on Scientific Computing}, 17\penalty0
  (3):\penalty0 773--776, 1996.

\bibitem[Ambikasaran(2013)]{sivaram13}
S.~Ambikasaran.
\newblock Fast algorithms for dense numerical linear algebra and applications.
\newblock \emph{PhD dissertation}, 2013.

\bibitem[Best(2002)]{best2002resonant}
S.~R Best.
\newblock On the resonant properties of the {Koch} fractal and other wire
  monopole antennas.
\newblock \emph{Antennas and Wireless Propagation Letters, IEEE}, 1\penalty0
  (1):\penalty0 74--76, 2002.

\bibitem[Borgani et~al.(1993)Borgani, Murante, Provenzale, and
  Valdarnini]{borgani1993multifractal}
S.~Borgani, G.~Murante, A.~Provenzale, and R.~Valdarnini.
\newblock Multifractal analysis of the galaxy distribution: reliability of
  results from finite data sets.
\newblock \emph{Physical Review E}, 47\penalty0 (6):\penalty0 3879, 1993.

\bibitem[Boschitsch et~al.(1999)Boschitsch, Fenley, and Olson]{bosch99}
A.~H Boschitsch, M.~O Fenley, and W.~K Olson.
\newblock A fast adaptive multipole algorithm for calculating screened coulomb
  (yukawa) interactions.
\newblock \emph{Journal of Computational Physics}, 151\penalty0 (1):\penalty0
  212--241, 1999.

\bibitem[Carrier et~al.(1988)Carrier, Greengard, and Rokhlin]{carrier88}
J.~Carrier, L.~Greengard, and V.~Rokhlin.
\newblock A fast adaptive multipole algorithm for particle simulations.
\newblock \emph{SIAM Journal on Scientific and Statistical Computing},
  9\penalty0 (4):\penalty0 669--686, 1988.

\bibitem[Chandramowlishwaran et~al.(2010)Chandramowlishwaran, Williams, Oliker,
  Lashuk, Biros, and Vuduc]{aparna10}
A.~Chandramowlishwaran, S.~Williams, L.~Oliker, I.~Lashuk, G.~Biros, and
  R.~Vuduc.
\newblock Optimizing and tuning the fast multipole method for state-of-the-art
  multicore architectures.
\newblock In \emph{Parallel \& Distributed Processing (IPDPS), 2010 IEEE
  International Symposium on}, pages 1--12. IEEE, 2010.

\bibitem[Cheng et~al.(1999)Cheng, Greengard, and Rokhlin]{greengard99}
H.~Cheng, L.~Greengard, and V.~Rokhlin.
\newblock A fast adaptive multipole algorithm in three dimensions.
\newblock \emph{Journal of Computational Physics}, 155\penalty0 (2):\penalty0
  468--498, 1999.

\bibitem[Cohen(1997)]{cohen1997fractal}
N.~Cohen.
\newblock Fractal antenna applications in wireless telecommunications.
\newblock In \emph{Electronics Industries Forum of New England, 1997.
  Professional Program Proceedings}, pages 43--49. IEEE, 1997.

\bibitem[Cohen(2000)]{cohen2000microstrip}
N.~Cohen.
\newblock Microstrip patch antenna with fractal structure, October~3 2000.
\newblock US Patent 6,127,977.

\bibitem[Darve(2000{\natexlab{a}})]{darve2000a}
E.~Darve.
\newblock The fast multipole method: numerical implementation.
\newblock \emph{Journal of Computational Physics}, 160\penalty0 (1):\penalty0
  195--240, 2000{\natexlab{a}}.

\bibitem[Darve(2000{\natexlab{b}})]{darve2000b}
E.~Darve.
\newblock The fast multipole method i: Error analysis and asymptotic
  complexity.
\newblock \emph{SIAM Journal on Numerical Analysis}, 38\penalty0 (1):\penalty0
  98--128, 2000{\natexlab{b}}.

\bibitem[Darve et~al.(2011)Darve, Cecka, and Takahashi]{darve11}
E.~Darve, C.~Cecka, and T.~Takahashi.
\newblock The fast multipole method on parallel clusters, multicore processors,
  and graphics processing units.
\newblock \emph{Comptes Rendus Mecanique}, 339\penalty0 (2):\penalty0 185--193,
  2011.

\bibitem[Fong and Darve(2009)]{fong09}
W.~Fong and E.~Darve.
\newblock The black-box fast multipole method.
\newblock \emph{Journal of Computational Physics}, 228\penalty0 (23):\penalty0
  8712--8725, 2009.

\bibitem[Gianvittorio and Rahmat-Samii(2002)]{gianvittorio2002fractal}
J.P. Gianvittorio and Y.~Rahmat-Samii.
\newblock Fractal antennas: A novel antenna miniaturization technique, and
  applications.
\newblock \emph{Antennas and Propagation magazine, IEEE}, 44\penalty0
  (1):\penalty0 20--36, 2002.

\bibitem[Goude and Engblom(2013)]{goude13}
A.~Goude and S.~Engblom.
\newblock Adaptive fast multipole methods on the gpu.
\newblock \emph{The Journal of Supercomputing}, 63\penalty0 (3):\penalty0
  897--918, 2013.

\bibitem[Greengard and Rokhlin(1987)]{greengard87}
L.~Greengard and V.~Rokhlin.
\newblock A fast algorithm for particle simulations.
\newblock \emph{Journal of computational physics}, 73\penalty0 (2):\penalty0
  325--348, 1987.

\bibitem[Greengard and Rokhlin(1989)]{greengard89}
L.~Greengard and V.~Rokhlin.
\newblock On the evaluation of electrostatic interactions in molecular
  modeling.
\newblock \emph{Chemica Scripta}, 29:\penalty0 139--144, 1989.

\bibitem[Greengard and Rokhlin(1997)]{greengard97}
L.~Greengard and V.~Rokhlin.
\newblock A new version of the fast multipole method for the laplace equation
  in three dimensions.
\newblock \emph{Acta numerica}, 6\penalty0 (1):\penalty0 229--269, 1997.

\bibitem[Gumerov et~al.(2003)Gumerov, Duraiswami, and Borovikov]{gumerov03}
N.~A Gumerov, R.~Duraiswami, and E.~A Borovikov.
\newblock Data structures, optimal choice of parameters, and complexity results
  for generalized multilevel fast multipole methods in $ d $ dimensions.
\newblock 2003.

\bibitem[Huang et~al.(2009)Huang, Jia, and Zhang]{huang09}
J.~Huang, J.~Jia, and B.~Zhang.
\newblock Fmm-yukawa: an adaptive fast multipole method for screened coulomb
  interactions.
\newblock \emph{Computer Physics Communications}, 180\penalty0 (11):\penalty0
  2331--2338, 2009.

\bibitem[Hutchinson(1981)]{hutchinson81}
J.~E Hutchinson.
\newblock Fractals and self-similarity.
\newblock \emph{Indiana University Mathematics Journal}, 30\penalty0
  (5):\penalty0 713--747, 1981.

\bibitem[Joyce et~al.(2000)Joyce, Anderson, Montuori, Pietronero, and
  Labini]{joyce2000}
M.~Joyce, P.~W. Anderson, M.~Montuori, L.~Pietronero, and F.~S. Labini.
\newblock Fractal cosmology in an open universe.
\newblock \emph{EPL (Europhysics Letters)}, 50\penalty0 (3):\penalty0 416,
  2000.

\bibitem[Lang et~al.(2005)Lang, Klaas, and de~Freitas]{lang2005empirical}
D.~Lang, M.~Klaas, and N.~de~Freitas.
\newblock Empirical testing of fast kernel density estimation algorithms.
\newblock \emph{UBC Technical report}, 2, 2005.

\bibitem[Lashuk et~al.(2012)Lashuk, Chandramowlishwaran, Langston, Nguyen,
  Sampath, Shringarpure, Vuduc, Ying, Zorin, and Biros]{zorin12}
I.~Lashuk, A.~Chandramowlishwaran, H.~Langston, T.~Nguyen, R.~Sampath,
  A.~Shringarpure, R.~Vuduc, L.~Ying, D.~Zorin, and G.~Biros.
\newblock A massively parallel adaptive fast multipole method on heterogeneous
  architectures.
\newblock \emph{CommuNiCAtioNs of the ACm}, 55\penalty0 (5):\penalty0 101--109,
  2012.

\bibitem[Mainieri(1993)]{ronnie93}
R.~Mainieri.
\newblock On the equality of {Hausdorff} and box counting dimensions.
\newblock \emph{Chaos An Interdisciplinary Journal of Nonlinear Science},
  3\penalty0 (2):\penalty0 119, 1993.

\bibitem[Mandelbrot(1983)]{mandelbrot83}
B.B. Mandelbrot.
\newblock \emph{The fractal geometry of nature}.
\newblock Macmillan, 1983.

\bibitem[Martinez and Jones(1990)]{martinez1990universe}
V.J. Martinez and B.J.T. Jones.
\newblock Why the universe is not a fractal.
\newblock \emph{Monthly Notices of the Royal Astronomical Society},
  242:\penalty0 517--521, 1990.

\bibitem[Michielssen and Boag(1996)]{michielssen1996multilevel}
E.~Michielssen and A.~Boag.
\newblock A multilevel matrix decomposition algorithm for analyzing scattering
  from large structures.
\newblock \emph{Antennas and Propagation, IEEE Transactions on}, 44\penalty0
  (8):\penalty0 1086--1093, 1996.

\bibitem[Nabors et~al.(1994)Nabors, Korsmeyer, Leighton, and White]{nabors94}
K.~Nabors, FT~Korsmeyer, FT~Leighton, and J.~White.
\newblock Preconditioned, adaptive, multipole-accelerated iterative methods for
  three-dimensional first-kind integral equations of potential theory.
\newblock \emph{SIAM Journal on Scientific Computing}, 15\penalty0
  (3):\penalty0 713--735, 1994.

\bibitem[Parr{\'o}n et~al.(2003)Parr{\'o}n, Romeu, Rius, and
  Mosig]{parron2003method}
J.~Parr{\'o}n, J.~Romeu, J.M. Rius, and J.R. Mosig.
\newblock Method of moments enhancement technique for the analysis of
  {Sierpinski} pre-fractal antennas.
\newblock \emph{Antennas and Propagation, IEEE Transactions on}, 51\penalty0
  (8):\penalty0 1872--1876, 2003.

\bibitem[Peebles(1980)]{james80}
P.~J.E. Peebles.
\newblock \emph{The large-scale structure of the universe}.
\newblock Princeton university press, 1980.

\bibitem[Pietronero(1987)]{pietronero1987fractal}
L.~Pietronero.
\newblock The fractal structure of the universe: correlations of galaxies and
  clusters and the average mass density.
\newblock \emph{Physica A: Statistical Mechanics and its Applications},
  144\penalty0 (2):\penalty0 257--284, 1987.

\bibitem[Ribeiro and Miguelote(1998)]{ribeiro1998fractals}
M.B. Ribeiro and A.Y. Miguelote.
\newblock Fractals and the distribution of galaxies.
\newblock \emph{Brazilian journal of physics}, 28\penalty0 (2):\penalty0
  132--160, 1998.

\bibitem[Schleicher(2007)]{dierk07}
D.~Schleicher.
\newblock {Hausdorff} dimension, its properties, and its surprises.
\newblock \emph{American Mathematical Monthly}, 114\penalty0 (6):\penalty0
  509--528, 2007.

\bibitem[Sevilgen et~al.(2000)Sevilgen, Aluru, and
  Futamura]{sevilgen2000provably}
Fatih~Erdogan Sevilgen, Srinivas Aluru, and Natsuhiko Futamura.
\newblock A provably optimal, distribution-independent parallel fast multipole
  method.
\newblock In \emph{Parallel and Distributed Processing Symposium, 2000. IPDPS
  2000. Proceedings. 14th International}, pages 77--84. IEEE, 2000.

\bibitem[Singh et~al.(1993)Singh, Holt, Hennessy, and Gupta]{singh93}
J.~P. Singh, C.~Holt, J.~L Hennessy, and A.~Gupta.
\newblock A parallel adaptive fast multipole method.
\newblock In \emph{Proceedings of the 1993 ACM/IEEE conference on
  Supercomputing}, pages 54--65. ACM, 1993.

\bibitem[Singh et~al.(1995)Singh, Holt, Totsuka, Gupta, and Hennessy]{singh95}
J.~P. Singh, C.~Holt, T.~Totsuka, A.~Gupta, and J.~Hennessy.
\newblock Load balancing and data locality in adaptive hierarchical n-body
  methods: Barnes-hut, fast multipole, and radiosity.
\newblock \emph{Journal of Parallel and Distributed Computing}, 27\penalty0
  (2):\penalty0 118--141, 1995.

\bibitem[Song et~al.(1997)Song, Lu, and Chew]{song97}
J.~Song, C.~C. Lu, and W.~C. Chew.
\newblock Multilevel fast multipole algorithm for electromagnetic scattering by
  large complex objects.
\newblock \emph{Antennas and Propagation, IEEE Transactions on}, 45\penalty0
  (10):\penalty0 1488--1493, 1997.

\bibitem[Strogatz(2001)]{strogatz94}
S.~Strogatz.
\newblock Nonlinear dynamics and chaos: with applications to physics, biology,
  chemistry and engineering.
\newblock 2001.

\bibitem[Vinoy et~al.(2001)Vinoy, Jose, Varadan, and Varadan]{vinoy2001hilbert}
K.J. Vinoy, K.A. Jose, V.K. Varadan, and V.V. Varadan.
\newblock Hilbert curve fractal antenna: A small resonant antenna for vhf/uhf
  applications.
\newblock \emph{Microwave and Optical Technology Letters}, 29\penalty0
  (4):\penalty0 215--219, 2001.

\bibitem[Ying et~al.(2004)Ying, Biros, and Zorin]{zorin03}
L.~Ying, G.~Biros, and D.~Zorin.
\newblock A kernel-independent adaptive fast multipole algorithm in two and
  three dimensions.
\newblock \emph{Journal of Computational Physics}, 196\penalty0 (2):\penalty0
  591--626, 2004.

\bibitem[Yokota and Barba(2011)]{yokota10}
R.~Yokota and L.~A Barba.
\newblock Treecode and fast multipole method for n-body simulation with cuda.
\newblock \emph{GPU Computing Gems Emerald Edition}, page 113, 2011.

\end{thebibliography}

\end{document}